\title{Inverse period mappings of $K3$ surfaces
and
a construction of modular forms for a lattice with the Kneser conditions}
\author{Atsuhira Nagano}
\documentclass[10pt]{article}
\usepackage{mathrsfs,bm,graphics,amsfonts,amsthm,amssymb,amsmath,amscd,booktabs}
\usepackage[dvips]{graphicx,color}
\def\bigzerou{\smash{\lower1.7ex\hbox{\b 0}}}

\newtheorem{thm}{Theorem}[section]
\newtheorem{df}{Definition}[section]
\newtheorem{lem}{Lemma}[section]
\newtheorem{prop}{Proposition}[section]
\newtheorem{rem}{Remark}[section]

\newtheorem{cor}{Corollary}[section]

\def\comment#1{{ }}
\makeatletter
  
      \@addtoreset{equation}{section}
      
\begin{document}
\maketitle

\begin{abstract}
We explicitly construct modular forms on a $4$-dimensional bounded symmetric domain of type $IV$
based on the variation of the Hodge structures of $K3$ surfaces.
We study   the ring of our modular forms.
Because of the Kneser conditions 
of the transcendental lattice of our family of $K3$ surfaces,
our modular group has a good arithmetic property.
Also, our results can be regarded as natural extensions of the result of \cite{CD} for  Siegel modular forms from the viewpoint of $K3$ surfaces. 
\end{abstract}

\footnote[0]{Running head: Inverse period mapping and modular forms}
\setlength{\baselineskip}{14 pt}

\section*{Introduction}
The purpose of this paper is to give an explicit construction of modular forms on a $4$-dimensional bounded symmetric domain of type $IV$ by using the variation of Hodge structures of weight $2$ for a family of $K3$ surfaces.
Our modular forms can be regarded as natural extensions of classical Siegel modular forms.

Let $U$ be the hyperbolic lattice of rank $2$ and $E_j$ $(j\in \{6,7,8 \})$ be the root lattice.
Let $L$ be the even unimodular lattice of signature $(3,19)$:  $L= U\oplus U \oplus U \oplus E_8 (-1 ) \oplus E_8(-1)$.
This lattice is called the $K3$ lattice,
since it is isomorphic to the lattice of the second homology group of $K3$ surfaces.
Let $M=U\oplus E_8(-1) \oplus E_6 (-1)$ be a lattice  of signature $(1,15)$.
We have the orthogonal complement 
\begin{align}\label{latticeA}
A=  U\oplus U  \oplus A_2 (-1)
\end{align}
 of $M$ in $L$.
 This is a lattice of signature $(2,4)$.
We note that this is the simplest lattice satisfying the Kneser conditions,
which induce good  properties of reflection groups (see Section 3.1).

Let us define the $4$-dimensional space 
$
\mathcal{D}_M= \{ [\xi] \in \mathbb{P} (A \otimes \mathbb{C}) \hspace{0.5mm} | \hspace{0.5mm}      (\xi, \xi) =0,  (\xi, \overline{\xi}) >0\}.
$
Let $\mathcal{D}$ be a connected component of $\mathcal{D}_M$.
Namely, $\mathcal{D}$ is a bounded symmetric domain of type $IV$.
Let $\Gamma = \tilde{O}^+ (A)$ be the subgroup of the
stable orthogonal group $\tilde{O}(A)$ for the lattice $A$ which preserves the connected component  $\mathcal{D}$
(for detail, see Section 2.2).
Then, $\Gamma $ acts on $\mathcal{D}$ discontinuously.
Let
$\mathcal{D}^*$ be the  connected component of the set $\{\xi \in A \otimes \mathbb{C} \hspace{0.5mm} | \hspace{0.5mm}  (\xi, \xi) =0,  (\xi, \overline{\xi}) >0\}$
which projects to $\mathcal{D}$.
We can consider the action of the group $\mathbb{C}^* \times \Gamma$ on $\mathcal{D}^*$.
Then, we  define the modular forms  for $\Gamma$ as follows.

\begin{df}\label{DfModularForm}
If a holomorphic function $f:\mathcal{D}^* \rightarrow \mathbb{C}$ 
given by $Z\mapsto f(Z)$ satisfies the conditions
\begin{itemize}

\item[(i)] $f(\lambda Z) = \lambda^{-k} f(Z) \quad (\text{for all } \lambda \in \mathbb{C}^*)$,

\item[(ii)] $f(\gamma Z) = \chi(\gamma) f(Z) \quad (\text{for all } \gamma \in \Gamma),$

\end{itemize}
where $k\in \mathbb{Z}$ and $\chi \in {\rm Char(\Gamma)={\rm Hom}(\Gamma,\mathbb{C}^*)}$,
then $f$ is called a modular form of weight $k$ and character $\chi$ for the group $\Gamma$.
\end{df}

Here, we note that modular forms in Definition \ref{DfModularForm}  give
natural extensions of the well-known  Siegel modular forms.
In fact, 
for the lattice 
$$A_S=U\oplus U \oplus A_1(-1),$$
we can similarly define the modular forms for the orthogonal group $\tilde{O}^+(A_{S})$.
The ring of such modular forms of even weight and the trivial character 
is generated by four modular forms of weight $4,6,10$ and $12$. 
They give  Siegel modular forms of degree $2$ of  even weight (see  \cite{GN} and \cite{CD}).

Let $\mathcal{A}_k(\Gamma,\chi)$ be the vector space of the modular forms of weight $k$ and character $\chi$.
We obtain the graded ring 
$$
\mathcal{A}(\Gamma) =\bigoplus_{k=0}^\infty \bigoplus_{\chi \in {\rm Char} (\Gamma)} \mathcal{A}_k(\Gamma,\chi)
$$
of modular forms.
In this paper, we shall determine the structure of this ring.

Due to the Kneser conditions,
 the modular group $\Gamma$
is generated by reflections 
for vectors with self-intersection $-2$.
Moreover, the character group of $\Gamma$ is very simple:
$
{\rm Char}(\Gamma) =\{{\rm id}, {\rm det}\}.
$
We shall determine the structure of this ring as the following theorem.

\begin{thm}\label{ThmMain} (see Corollary \ref{Corid} and Theorem \ref{ThmTotal}) 

(1)  
The ring  $\mathcal{A}(\Gamma,{\rm id})$ of modular forms of  character {\rm id} is isomorphic to  $\mathbb{C}[t_4,t_6,t_{10},t_{12},t_{18}]$.
Here, $t_k$ gives a modular form of weight $k$ and character ${\rm id}$.

(2) There is a modular form $s_{54}$  of weight $54$ and character ${\rm det}$.
Here, $s_{54}$ is given by $s_9 s_{45}$, where
\begin{align*}
\begin{cases}
&s_9^2 =t_{18},  \\
&s_{45}^2 = (\text{an irreducible homogeneous polynomial in } t_4,t_6,t_{10},t_{12} \text{ and } t_{18} \text{ of weight } $90$).
\end{cases}
\end{align*}
These relations determine the structure of the ring $\mathcal{A}(\Gamma)$.

\end{thm}

In this paper,
this  theorem will be proved via   period mappings of a family of $K3 $ surfaces.
This idea is natural,
because  moduli spaces of polarized  $K3$  surfaces are related to bounded symmetric domains of type $IV$.
Let us recall the following result by Clingher-Doran \cite{CD} (see also \cite{Kumar}, \cite{NS}).
They studied a family of  explicit elliptic $K3$ surfaces
$S_{CD} (\alpha,\beta,\gamma,\delta)$ (see (\ref{SCD})).
The singular fibres of the elliptic fibration for their family are corresponding to 
the root systems of type $E_8$ and $E_7$.
This structure gives a polarization for their $K3$ surfaces.
Also, they explicitly studied the Shioda-Inose structure on their $K3$  surfaces.
They constructed  Siegel modular forms  of degree $2$ 
by
inverse period mappings
for their $K3$ surfaces.
Such results of $K3$ surfaces can be effectively applied  to obtain important  models in arithmetic geometry
(for example, see \cite{EK}, \cite{NCF}, etc.).
By the way,
there exists a Siegel modular form of weight  $35$. 
According to Igusa \cite{I35},
this modular form satisfies a relation,
which determines the structure of the ring of Siegel modular forms.
It is remarkable that this modular form can be obtained by calculating  the ``discriminant'' of the $K3$ surfaces of \cite{CD}
(see Section 2.3).
 Thus, the ring of Siegel modular forms can be  studied via the geometric structure of appropriate $K3$ surfaces.

Thus, it is meaningful  
to construct modular forms via inverse period mappings of $K3$ surfaces.
In this paper,
we will study  a family of  $K3$ surfaces defined by 
$$
 z_0^2=y_0^3 + (t_4 x_0^4 + t_{10} x_0^3 )y_0 + (x_0^7 + t_6 x_0^6  + t_{12} x_0^5  + t_{18} x_0^4 ),
$$
where $x_0,y_0$ and $z_0$ are affine complex coordinates.
The singular fibres of this elliptic fibration are  corresponding to the root systems of type $E_8$ and $E_6$.
The family $S_{CD}(\alpha,\beta,\gamma,\delta)$ is a subfamily of our family
(see Proposition \ref{PropS(t)CD}).
So, Theorem \ref{ThmMain} gives a natural and non-trivial extension of the work \cite{CD} 
from the viewpoint of $K3$ surfaces (Table 1).
Especially,
the pair of the relations of Theorem \ref{ThmMain} (2) gives a counterpart of  above-mentioned Igusa's relation.
We remark that,
due to the result of \cite{Morrison},
our $K3$ surfaces do not admit the Shioda-Inose structure
and
the moduli space for our $K3$ surfaces
is beyond the moduli space of principally polarized abelian surfaces.

\begin{table}[h]
\center
\begin{tabular}{ccc}
\toprule
  & Results of \cite{CD}  & Our Results  \\
 \midrule
Root Systems for Elliptic Surfaces & $E_8$ and $E_7$ &  $E_8$ and $E_6$ \\ 
 Transcendental Lattice  & $U\oplus U\oplus A_1(-1)$ & $U\oplus U\oplus A_2(-1)$  \\
 Lie Group for Symmetric Domain &  $SO_0(2,3)$  & $SO_0(2,4)$\\
 Inverse Period Mapping &  Classical Siegel Modular Forms  & Modular Forms for $\Gamma$    \\
Weights of Parameters &  $4,6,10,12$ &  $4,6,10,12,18$ \\
Weights from Discriminants &  $5,30$ & $9,45$ \\
 \bottomrule
\end{tabular}
\caption{The results of \cite{CD} and our results}
\end{table}

In Section  1 and Section 2,
we will study the period mapping for our $K3$ surfaces precisely.
In Section 1,
we will investigate the structure of elliptic surfaces for our family.
Here, we will determine the structure of the N\'eron-Severi lattice and the transcendental lattice of a generic member of our family. 
In Section 2,
we will obtain an explicit expression of our period mapping.
The Torelli type theorem and the surjectivity of period mappings for $K3$ surfaces
are very important for our study.
According to Corollary \ref{CorPer},
 our period mapping gives an isomorphism 
 between the parameter space of our family of $K3$ surfaces
 and 
 the quotient space $\mathcal{D}/\Gamma$ for the symmetric domain $\mathcal{D}$.
This 
guarantees that
the inverse period mappings give 
 modular forms for $\Gamma$.
In Section 3,
we will see the properties of latices with the  Kneser conditions.
These properties  enable us 
to study our period mapping precisely.
Especially,
we can study
the structure of the branch loci
coming from the action of the group $\Gamma$ on $\mathcal{D}$.
In Section 4,
we will show that
 the polynomial ring $\mathbb{C}[t_4,t_6,t_{10},t_{12},t_{18}]$
in the parameters of our $K3$ surfaces gives the ring 
  of  modular forms  of the trivial character.
In the proof,  we will consider the structure of line bundles over the modular variety. 
In Section 5,
we will show that there exist 
 $s_{54} = s_9 s_{45} $ of  character ${\rm det}$.
The structure of the ring $\mathcal{A}(\Gamma)$ is determined by the equations for this modular form.
In this argument, 
we will use
some techniques of canonical orbibundles over  orbifolds,
 referring to the  work  of Hashimoto-Ueda \cite{HashimotoUeda}.

At the end of the introduction,
we note two merits of our results for applications  (for detail, see the end of the last section).

We can characterize our results from the viewpoint of relations between $K3$ surfaces and algebraic combinatorics.
Finite unitary reflection groups,
that were classified by Shepherd-Todd \cite{ST},
are very important.
There are several results  about periods of $K3$ surfaces related to Shepherd-Todd groups.
This paper gives a new non-trivial result in this direction (Table 2).

\begin{table}[h]
\center
\begin{tabular}{ccccc}
\toprule
 Shepherd-Todd Group  & Weights of Invariants  &  $K3$ Surfaces and Periods   \\
 \midrule
 No. 8 & $8,12$ &  \cite{IS}  \\ 
No. 23  &$2,6,10$ &  \cite{NTheta} \\
No. 31 & $8,12,20,24$   &\cite{CD} \\
 No. 33 &  $4,6,10,12,18$  & This Paper  \\
 \bottomrule
\end{tabular}
\caption{Relations between combinatorics and  $K3$ surfaces}
\end{table}

Also,
we note that
our modular forms not only give natural extensions of the classical Siegel modular forms,
but also
give modular forms on the  bounded symmetric domain of type $I_{2,2}$,
which corresponds to the Hermitian form of signature $(2,2)$.
This  is coming from the isomorphism 
 $SO_0 (2,4) \simeq SU(2,2)$
 of Lie groups.
Here, we need to recall the  pioneering work
of Matsumoto-Sasaki-Yoshida \cite{MSY} and Matsumoto \cite{Matsumoto}.
They studied
another family of 
 $K3$ surfaces  
and obtained  modular forms on $I_{2,2}$.
However,  our motivation is different from their motivation
and our modular forms are different from their modular forms.
Each result has good features respectively
(see Section 3.2). 
The author would like to point out that
our modular group $\Gamma$  is 
 generated by reflections for vectors with self-intersection $-2$,
whereas the modular group in \cite{MSY} is not so.
In fact, our argument of Section 4 and Section 5 is based on this property of reflections.
Moreover, this characteristic of our modular group
 let us expect non-trivial applications of our modular forms.

\section{Family of elliptic $K3$ surfaces $S(t)$}
 
  In this section,
  we define our $K3$ surfaces.
  They are given by hypersurfaces in a weighted projective space.
 
 \subsection{Hypersurfaces $S(t)$}
 Let $t=(t_4,t_6,t_{10},t_{12},t_{18})\in \mathbb{C}^5-\{0\}.$
 We shall consider the hypersurfaces
 \begin{align}\label{S(t)}
 S(t): z^2=y^3 + (t_4 x^4 w^4 + t_{10} x^3 w^{10})y + (x^7 + t_6 x^6 w^6 + t_{12} x^5 w^{12} + t_{18} x^4 w^{18})
 \end{align}
 of weight $42$
 in the weighted projective space 
 $\mathbb{P}(6,14,21,1)={\rm Proj} (\mathbb{C} [x,y,z,w])$.
 There is an action of the multiplicative group $\mathbb{C}^*$ 
 on $\mathbb{P}(6,14,21,1)$
 ($\mathbb{C}^5-\{0\}$, resp.)
 given by
 $(x,y,z,w) \mapsto (x,y,z,\lambda^{-1} w)$
 ($t=(t_4,t_6,t_{10},t_{12},t_{18}) \mapsto \lambda \cdot t =(\lambda^4 t_4, \lambda^6 t_6, \lambda^{10 } t_{10}, \lambda^{12 } t_{12}, \lambda^{18} t_{18})$, resp.)
  for $\lambda\in \mathbb{C}^*$.
Then, the surface $S(t)$ is invariant under the above action of $\mathbb{C}^*$.
 Hence,
 from the family
 $$
 \{S(t) \hspace{0.5mm} | \hspace{0.5mm} t\in \mathbb{C}^5 -\{0\} \} \rightarrow \mathbb{C}^5 -\{0\},
 $$
 we naturally obtain the family
 $$
\{S ([t]) \hspace{0.5mm} | \hspace{0.5mm} [t] \in \mathbb{P}(4,6,10,12,18) \} \rightarrow \mathbb{P}(4,6,10,12,18).
 $$
 Here, 
 the point of $\mathbb{P}(4,6,10,12,18)$ corresponding to $t=(t_4,t_6,t_{10},t_{12},t_{18}) \in \mathbb{C}^5 -\{0\}$
 is denoted by
 $[t]=(t_4:t_6:t_{10}:t_{12}:t_{18})$.
 
 Set
 \begin{align}
 T= \{[t] \in \mathbb{P}(4,6,10,12,18) \hspace{0.5mm} | \hspace{0.5mm} S([t]) \text{ gives an  elliptic } K3 \text{ surface} \}.
\end{align}
 By an explicit calculation as in \cite{NS} Section 3, we can see that
 \begin{align}\label{TDef}
 \mathbb{P}(4,6,10,12,18) - T = \{[t] \hspace{0.5mm} | \hspace{0.5mm} t_{10}=t_{12}=t_{18} =0 \}.
 \end{align}
 Namely, $T$ is an analytic subset of codimension $3$ in the weighted projective space $\hat{T}=\mathbb{P}(4,6,10,12,18)$.
  Anyway, we have the family
 $$
 \{S([t]) \hspace{0.5mm} | \hspace{0.5mm} [t]\in T \} \rightarrow T
 $$
 of elliptic $K3$ surfaces.
Here, let us consider the $\mathbb{C}^*$-bundle $T^*\rightarrow T$ naturally coming from the above action $t\mapsto \lambda\cdot t$ of $\mathbb{C}^*$.
We  also have the family 
 $$
 \{S(t) \hspace{0.5mm} | \hspace{0.5mm} t \in T^*\} \rightarrow T^*
 $$
 of $K3$ surfaces.
 The above action of $\mathbb{C}^*$ on $(x,y,z,w)$ and $t $ gives an isomorphism
 \begin{align}\label{lambdamap}
 \lambda : S(t) \mapsto S(\lambda\cdot t)
 \end{align}
 for any $\lambda\in \mathbb{C}^*.$
 
 There exists the unique holomorphic $2$-form $\omega_t$ on the $K3$ surface $S(t)$ up to a constant factor.
 Then,
 we can take a holomorphic family
$\{\omega_t\}_{t\in T^*}$,
where $\omega_t$ is the unique holomorphic $2$-form on the $K3$ surface $S(t)$.
 We can see that the above action of $\mathbb{C}^*$ 
 transforms $\omega_t$ into $\lambda^{-1} \omega_t$.
 So, the isomorphism $\lambda$ of (\ref{lambdamap})
 gives the correspondence
 \begin{align}\label{omegalambda}
 \lambda^{*} \omega_{\lambda \cdot t} = \lambda^{-1} \omega_t.
 \end{align}
 From (\ref{omegalambda}),
 we have
 the family $\{\omega_{[t]}\}_{[t]\in T}$
 of holomorphic $2$-forms
 for the family $\{S([t]) \hspace{0.5mm} | \hspace{0.5mm} [t]\in T \}$ of $K3$ surfaces.

  \subsection{Elliptic fibration and periods from $S$-marking }
 
In this subsection,
we will obtain a marking coming from the structure of  elliptic $K3$ surfaces.
Also, we will define the period mapping for our marked $K3$ surfaces.

 Set $x_0=\frac{x}{w^6}$.
 Let 
 \begin{align}\label{R(x0t)}
 R(x_0,t)&=\frac{1}{w^{84} x_0^8} \text{( the discriminant of the right hand side of (\ref{S(t)}) in } y) \notag \\
 & = 27 x_0^6+ 54 t_6 x_0^5+ (54 t_{12} + 4 t_4^3 + 27 t_6^2) x_0^4+(54 t_{18} + 12 t_4^2 t_{10}  + 54t_6 t_{12} ) x_0^3 \notag \\
 &\quad +(27 t_{12}^2 + 12 t_4 t_{10}^2  + 54 t_6 t_{18} ) x_0^2+(4 t_{10}^3 + 54 t_{12} t_{18}) x_0  + 27 t_{18}^2 . 
  \end{align}
 Set
 \begin{align*}
 \begin{cases}
& g_2^\vee (x_0,t)=\frac{1}{w^{28} x_0^3} (t_4 x^4 w^4 + t_{10} x^3 w^{10}) =t_4 x_0 + t_{10} , \\
& g_3^\vee (x_0,t)=\frac{1}{w^{42} x_0^4} (x^7 + t_6 x^6 w^6 + t_{12} x^5 w^{12} + t_{18} x^4 w^{18}) =x_0^3 +t_6 x_0^2 + t_{12} x_0 +t_{18}.
\end{cases}
 \end{align*}
 Let $r(t)$ be the resultant of the polynomials $g_2^\vee (x_0,t) $ and $g_3^\vee (x_0,t)$ in $x_0$:
\begin{align*}
r(t) = t_{10}^3 +t_{4}^2 t_{10} t_{12}  - t_4^3 t_{18}  - t_4 t_6 t_{10}^2,
 \end{align*}
 Then,
 by an explicit calculation,
  we can see that the discriminant of the polynomial $R(x_0,t)$ in $x_0$ is given by
 $r(t)^3 d_{90} (t)$,
 where $d_{90} (t)$ is given by an irreducible  homogeneous polynomial of weight $90$ in $t$:
 \begin{align}\label{d_{90} (t)}
 d_{90} (t) = & 3125 t_{10}^9 + 11664 t_{10}^3 t_{12}^5 + 151875 t_{10}^6 t_{12} t_{18} + 
  314928 t_{12}^6 t_{18} + 1968300 t_{10}^3 t_{12}^2 t_{18}^2 + 
  4251528 t_{12}^3 t_{18}^3\notag\\
  &+ 14348907 t_{18}^5 + 16200 t_4 t_{10}^5 t_{12}^3  + 
  472392 t_4 t_{10}^2 t_{12}^4 t_{18}  - 273375 t_4 t_{10}^5 t_{18}^2  - 
  5314410 t_4 t_{10}^2 t_{12} t_{18}^3  \notag \\
  &+ 4125  t_4^2 t_{10}^7 t_{12}  + 
  108135 t_4^2 t_{10}^4 t_{12}^2 t_{18}  - 1259712 t_4^2 t_{10} t_{12}^3 t_{18}^2  + 
  4251528 t_4^2 t_{10} t_{18}^4  + 864 t_4^3 t_{10}^3 t_{12}^4  \notag \\ 
  &
   - 
  3525 t_4^3 t_{10}^6 t_{18}  + 23328 t_4^3 t_{12}^5 t_{18} 
  - 
  378108t_4^3 t_{10}^3 t_{12} t_{18}^2  + 1102248 t_4^3 t_{12}^2 t_{18}^3
  + 
  888 t_4^4 t_{10}^5 t_{12}^2   \notag \\
  & + 26568 t_4^4 t_{10}^2 t_{12}^3 t_{18}  + 
  227448 t_4^4 t_{10}^2 t_{18}^3  + 16  t_4^5 t_{10}^7 - 456 t_4^5 t_{10}^4 t_{12} t_{18}  - 
  85536 t_4^5 t_{10} t_{12}^2 t_{18}^2\notag \\
  &  + 16 t_4^6 t_{10}^3 t_{12}^3  + 
  432 t_4^6 t_{12}^4 t_{18}  - 1056 t_4^6 t_{10}^3 t_{18}^2  + 62208 t_4^6t_{12} t_{18}^3  + 
  16 t_4^7 t_{10}^5 t_{12} + 480 t_4^7 t_{10}^2 t_{12}^2 t_{18}  \notag \\
  &  - 16 t_4^8 t_{10}^4 t_{18} - 
  1536 t_4^8 t_{10} t_{12} t_{18}^2  + 1024 t_4^9 t_{18}^3  - 13500 t_6 t_{10}^6 t_{12}^2  - 
  481140 t_6 t_{10}^3 t_{12}^3 t_{18} - 2834352   t_6 t_{12}^4 t_{18}^2 \notag \\ 
  &  - 
  1476225 t_6 t_{10}^3 t_{18}^3 - 19131876 t_6 t_{12} t_{18}^4  - 5625 t_4 t_6 t_{10}^8 - 
  200475  t_4 t_6 t_{10}^5 t_{12} t_{18} - 236196 t_4 t_6 t_{10}^2 t_{12}^2 t_{18}^2 \notag \\
  &  - 
  2592  t_4^2 t_6 t_{10}^4 t_{12}^3 - 69984 t_4^2 t_6 t_{10} t_{12}^4 t_{18}  + 
  422820 t_4^2 t_6  t_{10}^4 t_{18}^2 + 944784  t_4^2 t_6 t_{10} t_{12} t_{18}^3 - 
  3420 t_4^3 t_6 t_{10}^6 t_{12} \notag \\
  & - 107460 t_4^3 t_6 t_{10}^3 t_{12}^2 t_{18}  - 
  174960 t_4^3 t_6t_{12}^3 t_{18}^2  - 1889568  t_4^3 t_6 t_{18}^4 + 
  2772 t_4^4 t_6t_{10}^5 t_{18}  + 314928 t_4^4 t_6 t_{10}^2 t_{12} t_{18}^2 \notag \\
  & - 
  186624 t_4^5 t_6 t_{10} t_{18}^3  - 16 t_4^6 t_6 t_{10}^6  - 
  576 t_4^6 t_6 t_{10}^3 t_{12} t_{18}  - 3456 t_4^6 t_6 t_{12}^2 t_{18}^2  + 
  1152 t_4^7 t_6 t_{10}^2 t_{18}^2  - 5832 t_6^2 t_{10}^3 t_{12}^4  \notag \\
  &- 
  10125 t_6^2 t_{10}^6 t_{18}  - 157464 t_6^2 t_{12}^5 t_{18}  - 
  295245 t_6^2 t_{10}^3 t_{12} t_{18}^2  + 5314410 t_6^2 t_{12}^2 t_{18}^3  - 
  5670 t_4 t_6^2 t_{10}^5 t_{12}^2 \notag \\
  & - 170586 t_4 t_6^2 t_{10}^2 t_{12}^3 t_{18}  + 
  3188646 t_4 t_6^2  t_{10}^2 t_{18}^3 + 2700 t_4^2 t_6^2 t_{10}^7  + 
  101898 t_4^2 t_6^2  t_{10}^4 t_{12} t_{18} + 
  1102248 t_4^2 t_6^2 t_{10} t_{12}^2 t_{18}^2  \notag \\
  & + 216t_4^3 t_6^2 t_{10}^3 t_{12}^3  + 
  5832  t_4^3 t_6^2 t_{12}^4 t_{18}  - 195048  t_4^3 t_6^2 t_{10}^3 t_{18}^2  + 
  216 t_4^4 t_6^2 t_{10}^5 t_{12}  + 6480  t_4^4 t_6^2 t_{10}^2 t_{12}^2 t_{18} \notag \\
  &- 
  216  t_4^5 t_6^2 t_{10}^4 t_{18} - 20736 t_4^5 t_6^2 t_{10} t_{12} t_{18}^2  + 
  20736 t_4^6 t_6^2 t_{18}^3  + 6075 t_6^3 t_{10}^6 t_{12}  + 
  219429  t_6^3 t_{10}^3 t_{12}^2 t_{18} \notag \\
  & + 1338444 t_6^3 t_{12}^3 t_{18}^2  + 
  4251528 t_6^3 t_{18}^4  + 1215 t_4 t_6^3  t_{10}^5 t_{18} - 
  393660 t_4 t_6^3 t_{10}^2 t_{12} t_{18}^2  - 1259712 t_4^2 t_6^3 t_{10} t_{18}^3  \notag \\
&  - 
  216 t_4^3 t_6^3 t_{10}^6  - 7776 t_4^3 t_6^3 t_{10}^3 t_{12} t_{18}  - 
  46656 t_4^3 t_6^3 t_{12}^2 t_{18}^2  + 15552  t_4^4 t_6^3  t_{10}^2 t_{18}^2+ 
  729 t_6^4 t_{10}^3 t_{12}^3  + 19683 t_6^4 t_{12}^4 t_{18}  \notag \\
  &- 
  8748 t_6^4 t_{10}^3 t_{18}^2  - 2834352 t_6^4 t_{12} t_{18}^3  + 
  729 t_4 t_6^4 t_{10}^5 t_{12}  + 21870  t_4 t_6^4 t_{10}^2 t_{12}^2 t_{18} - 
  729 t_4^2 t_6^4 t_{10}^4 t_{18} \notag \\
  &  - 69984 t_4^2 t_6^4 t_{10} t_{12} t_{18}^2  + 
  139968   t_4^3 t_6^4 t_{18}^3 - 729 t_6^5 t_{10}^6  - 
  26244  t_6^5 t_{10}^3 t_{12} t_{18} - 157464 t_6^5 t_{12}^2 t_{18}^2  \notag \\
  &+ 
  52488 t_4 t_6^5 t_{10}^2 t_{18}^2  + 314928 t_6^6 t_{18}^3.
 \end{align}
 
 Here, we can compute the singular fibres on these divisors corresponding to the discriminant of $R(x_0,t)$
 (for detail, see \cite{NS} Section 3.1 and \cite{HashimotoUeda} Section 6).
 If $[t]$ is on the divisor $\{[t]\in T \hspace{0.5mm} | \hspace{0.5mm} r(t) =0\}$,
 then both $g_2^\vee (x_0,t)$ and $g_3^\vee (x_0,t)$ is equal to $0$ for some $x_0$.
 On such a point,
 the singular fibre for the elliptic surface $S([t])$ is of  Kodaira type $II$.
 We remark that this type of singularity does not acquire any new singularity.
 On the other hand,
 for a generic point $[t]$ of the divisor $\{[t]\in T\hspace{0.5mm} | \hspace{0.5mm} d_{90} (t)=0 \}$,
 two of singular fibres of type $I_1$ of (\ref{Kodaira})
 collapse into a singular fibre of Kodaira type  $I_2$.
 We note that this type of singularity acquires an $A_1$-singularity.

 Set $\mathcal{T} = T - \{[t]\in T \hspace{0.5mm} | \hspace{0.5mm} d_{90} (t) =0 \text{ or } t_{18} =0 \}. $
 For a generic point $[t] \in \mathcal{T},$
 $S([t])$ gives an elliptic $K3$ surface $\pi:S([t]) \rightarrow \mathbb{P}^1(\mathbb{C})$ whose singular fibres are illustrated in Figure 1.
Namely,  singular fibres are of Kodaira type
 \begin{align}\label{Kodaira}
 II^* + IV^* + 6 I_1.
 \end{align}
 Here, $\pi^{-1} (\infty)$ is of type $II^*$ and $\pi^{-1} (0)$ is of type $IV^*.$
 The general fibre $F$ satisfies the following:
 \begin{align} \label{F-linear}
  F \text{ is linearly equivalent to } a_0 + 2 a_1 + 3 a_2 + 4 a_3 + 5 a_4 + 6 a_5 + 3 a_6 + 4 a_7 + 2 a_8.
 \end{align}

\begin{figure}[h]
\center
\includegraphics[scale=0.8]{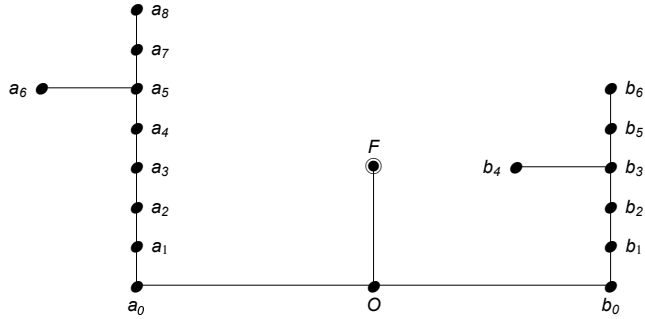}
\caption{Elliptic fibres of $S(t)$}
\end{figure}

 Take $[t]\in \mathcal{T}$.
Let us consider the sublattice of
$ H_2 (S([t]),\mathbb{Z})$ 
generated by
\begin{align}\label{generators}
F,O,
a_1, \cdots, a_8,b_1,\cdots, b_6.
\end{align}
 This lattice is a sublattice of ${\rm NS}(S([t]))$ which is isomorphic to the lattice 
 \begin{align}\label{latticeM}
 M=U\oplus E_8(-1) \oplus E_6(-1)
 \end{align}
  of rank $16$.
So, 
we have an isometry
\begin{align*}
\psi: H_2 (S([t]),\mathbb{Z}) \rightarrow L 
\end{align*} 
 such that $\psi^{-1} (M) \subset {\rm NS}(S([t]))$.
 We take
 $\gamma_7,\cdots,\gamma_{22}\in L$
 as
 $$
 \gamma_7=\psi(F),\gamma_8=\psi(O),\gamma_{8+j} =\psi(a_j) , \gamma_{16+k}=\psi(b_k) \quad(j\in \{1,\cdots,8\}, k\in\{1,\cdots,6\}).
 $$
 Then, $\gamma_7,\cdots,\gamma_{22}$ generate the lattice $M$.
 Here,
 since
 $\hspace{0.5mm} | \hspace{0.5mm}{\rm det}(M)\hspace{0.5mm} | \hspace{0.5mm}=3$
 is a prime number,
 $M$ is a primitive sublattice of $L$.
 Therefore,
 there exist $\gamma_1,\cdots,\gamma_6\in L$
 such that
 \begin{align}\label{gammagenerator}
\{ \gamma_1,\cdots,\gamma_6,\gamma_7,\cdots,\gamma_{22}\}
 \end{align}
 gives a basis of $L$.
 Let $\delta_1,\cdots,\delta_{22}$
 be the dual basis of the basis in (\ref{gammagenerator})
 with respect to the intersection form
 given by the unimodular lattice $L$.
 Here, we note that $\delta_1,\cdots,\delta_6$ generate the lattice 
 $A$ of (\ref{latticeA}).

 Let $S_0 = S([t_0])$ for $[t_0] \in \mathcal{T}$
 be a reference surface.
 Take a sufficiently small neighborhood $U$ of $[t_0]$ in $\mathcal{T}$
 such that there exists a topological trivialization
 $\tau:\{S([t])\hspace{0.5mm} | \hspace{0.5mm}[t]\in U\} \rightarrow S_0 \times U$.
 Let
 $\beta: S_0 \times U \rightarrow S_0$
 be the canonical projection.
 Put
 $r=\beta \circ \tau.$
 Then,
 $r'_{[t]} = r | _{S([t])}$
 gives a $\mathcal{C}^\infty$-isomorphism of complex surfaces.
 For any $[t]\in U,$
 we have an isometry
 $\psi_{[t]} : H_2(S([t]),\mathbb{Z}) \rightarrow L$
 given by
 $\psi_{[t]} = (r'_{[t]})_*$.
 We call this isometry the $S$-marking on $U$.
 By an analytic continuation along an arc $\alpha \subset \mathcal{T}$,
 we can define the $S$-marking on $\mathcal{T}$.
 We remark that this depends on the choice of $\alpha.$
 The $S$-marking preserves the N\'eron-Severi lattice.
 So, we obtain the local period mapping 
 \begin{align}\label{localperiod}
 \Phi_1 : \mathcal{T} \rightarrow \mathbb{P}^5 (\mathbb{C})
 \end{align}
 given by
 \begin{align}\label{S-Per}
 [t] \mapsto [\xi] = \Big( \int_{\psi_{[t]}^{-1} (\gamma_1)} \omega_{[t]} :\cdots: \int_{\psi_{[t]}^{-1} (\gamma_6) } \omega_{[t]}  \Big),
 \end{align}
  where $\omega_{[t]}$ is the unique holomorphic $2$-form on $S([t])$ up to a constant factor
  and $\gamma_1,\cdots,\gamma_6\in L$ are given by (\ref{gammagenerator}).
  We call the pair $(S([t]), \psi_{[t]})$ an $S$-marked $K3$ surface.

\begin{df}\label{DfS-marking}
 Suppose there are two $S$-marked $K3$ surfaces $(S([t_1]),\psi_1)$ and $(S([t_2]),\psi_2)$.
 We  say that $(S([t_1]),\psi_1)$ and $(S([t_2]),\psi_2)$ are equivalent (isomorphic, resp.) 
 if there exists a biholomorphic mapping $f:S([t_1] ) \rightarrow S([t_2])$ such that
 $(\psi_2 \circ f_* \circ \psi_1^{-1})\hspace{0.5mm} | \hspace{0.5mm}_{M} = {\rm id}_{M}$
 ($\psi_2 \circ f_* \circ \psi_1^{-1} = {\rm id}_{L}$, resp.).
 \end{df}

 Since the image of $\Phi_1$ satisfies the Riemann-Hodge relation
 \begin{align}\label{R-Hrel}
 \xi A {}^t \xi =0, \quad \quad \xi A{}^t\overline{\xi} >0,
 \end{align}
  the image of $\Phi_1$ should be contained in the 
  $4$-dimensional space 
  \begin{align}\label{DM}
\mathcal{D}_M =  \{[\xi] \in \mathbb{P}^5 (\mathbb{C}) \hspace{0.5mm} | \hspace{0.5mm} \xi \text{ satisfies (\ref{R-Hrel})} \}.
  \end{align}
 This space has two connected components.
  Let $\mathcal{D}$ be one of the connected components.
  This is a bounded symmetric domain  of type $IV$.

  \subsection{Techniques of elliptic surfaces and $S$-marked $K3$ surfaces}
  In this subsection,
  we shall see a  relation
  between elliptic surfaces and $S$-marked $K3$ surfaces.

  Let $\pi:S\rightarrow \mathbb{P}^1 (\mathbb{C})$ be an elliptic surface with a general fibre $F$.
  Then,
  by virtue of the Riemann-Roch theorem, 
 we can see that
 $\pi $ is the unique elliptic fibration up to ${\rm Aut}(\mathbb{P}^1 (\mathbb{C}))$
 such that $F$ is a general fibre.
 So, we have the following definition.
 
 \begin{df}
 Let $(S_1,\pi_1,\mathbb{P}^1(\mathbb{C}))$ and $(S_2,\pi_2,\mathbb{P}^1(\mathbb{C}))$
 are two elliptic surfaces.
 If there exist a biholomorphic mapping $f: S_1\rightarrow S_2$
 and
 $\varphi \in {\rm Aut} (\mathbb{P}^1 (\mathbb{C}))$
 such that
 $\varphi\circ \pi_1 = \pi_2 \circ f$,
 then we say that
 $(S_1,\pi_1,\mathbb{P}^1(\mathbb{C}))$ and $(S_2,\pi_2,\mathbb{P}^1(\mathbb{C}))$
 are isomorphic as elliptic surfaces.
 \end{df}

 For an elliptic curve given by the Weierstrass form
 $
 z_0^2= y_0^3 - g_2 (x_0) y_0 -g_3(x_0),
 $
 the quotient $j(x_0)=\frac{g_2^3(x_0)}{4 g_2^3(x_0) -27 g_3^2 (x_0)}$ of the coefficients of the elliptic surface is called the $j$-invariant.
 Let $(S_1,\pi_1,\mathbb{P}^1(\mathbb{C}))$  ($(S_2,\pi_2,\mathbb{P}^1(\mathbb{C}))$, resp.)
 is an elliptic surface
 given by the Weierstrass form with the $j$-invariant $j_1(x_0)$ ($j_2(x_0)$, resp.),
 then there exists $\varphi \in {\rm Aut} (\mathbb{P}^1 (\mathbb{C}))$
 such that
 $\pi_1^{-1} (x_0)$ and $\pi_2^{-1} (\varphi (x_0))$ are singular fibres of the same type 
 and
 $j_2 \circ \varphi =j_1$ holds.
 
 For our $K3$ surface $S([t])$ of (\ref{S(t)}),
 $\pi:(x,y,z,w) \mapsto (x,w)$ gives  a natural elliptic fibration.
 Let $x_0=\frac{x}{w^6}$ be an affine coordinate.

 \begin{lem}\label{Lemma1}
 For $[t_1]$ and $[t_2]\in T$,
 we suppose that
 $(S([t_1]),\pi_1,\mathbb{P}^1 (\mathbb{C}))$ is isomorphic to  $(S([t_2]),\pi_2,\mathbb{P}^1 (\mathbb{C}))$
 as elliptic surfaces.
 Then, it holds that $[t_1]=[t_2]$.
 \end{lem}

 \begin{proof}
 We suppose that 
 a biholomorphic mapping $f:S([t_1]) \rightarrow S([t_2])$
 gives an equivalence of elliptic surfaces.
 Then,
 there exists  $\varphi \in {\rm Aut}(\mathbb{P}^1 (\mathbb{C}))$
 such that $\varphi \circ \pi_1 = \pi_2 \circ f.$
 Here, for $j=1,2$, 
 $\pi_j^{-1} (0)$ ($\pi_j^{-1} (\infty)$, resp.) is a singular fibre  of Kodaira type $II^*$ ($IV^*$, resp.).
 This implies that $\varphi \in {\rm Aut} (\mathbb{P}^1 (\mathbb{C}))$
 should be given by the mapping $x_0 \mapsto \lambda x_0$ for  some $\lambda \in \mathbb{C}^*.$
Let us consider the discriminant $R_j(x_0,t)$ $(j=1,2)$ of the right hand side of the elliptic surface $S([t_j])$.
They are given in the form of  (\ref{R(x0t)}), that are  polynomials in $x_0$ of degree $6$.
The six roots of each polynomial give the six images of the singular fibres of type $I_1$ on $\mathbb{P}^1 (\mathbb{C})=x_0\text{-sphere}$.
The roots of $R_1(x_0,t)$ should be sent to those of $R_2(x_0,t)$ by $\varphi.$
Hence, by observing the coefficients of $R_1(x_0,t)$ and $R_2(x_0,t)$,
we can see that  $[t_1]=[t_2]$.
 \end{proof}

 \begin{lem}\label{LemmaKey}
 For $[t_1]$ and $[t_2]\in \mathcal{T}$,
 two $S$-marked $K3$ surfaces $(S([t_1]),\psi_1)$ and $(S([t_2]),\psi_2)$ are equivalent 
 if and only if
 $[t_1]=[t_2].$
 \end{lem}
 
 \begin{proof}
 For $[t_1]$ and $[t_2] \in \mathcal{T}$,
 we suppose that
 there exists a biholomorphic mapping $f:S([t_1]) \rightarrow S([t_2])$
 which gives an equivalence of $S$-marked $K3$ surfaces.
 So, it follows  that
 $(\psi_2 \circ f_* \circ \psi_1^{-1})\hspace{0.5mm} | \hspace{0.5mm}_M = {\rm id}\hspace{0.5mm} | \hspace{0.5mm}_M$.
 This implies that
 $f_* (F_1)=F_2$ holds,
  where
  $F_j$ is a general fibre  for the elliptic surface $S([t_j])$ $(j=1,2)$.
  Then,
 $F_2$
 is a general fibre not only for the elliptic fibration $\pi_2$
 but also for another elliptic fibration $\pi_1 \circ f^{-1}$.
 Therefore, we have $\pi_2 = \pi_1 \circ f^{-1}$ holds up to ${\rm Aut} (\mathbb{P}^1 (\mathbb{C}))$.
 Hence, we proved that 
 two $S$-marked $K3$ surfaces $(S([t_1]),\psi_1)$ and $(S([t_2]),\psi_2)$ are equivalent 
 if and only if
 there exists an isomorphism of elliptic surfaces
 between $(S([t_1]),\pi_1,\mathbb{P}^1(\mathbb{C}))$ and $(S([t_2]),\pi_2,\mathbb{P}^1 (\mathbb{C}))$.
 Due to Lemma \ref{Lemma1},
 the assertion is proved.
 \end{proof}

 \subsection{Picard number}

 \begin{thm}
 For a generic point $[t]\in \mathcal{T},$
 the Picard number of $S([t])$  is equal to $16$.
 \end{thm}

 \begin{proof}
 Since $M$ of (\ref{latticeM}) is isomorphic to a sublattice of ${\rm NS}(S([t]))$ for $[t] \in \mathcal{T}$,
 it is apparent that ${\rm rank}({\rm NS}(S([t]))) \geq 16$.
 Take a sufficiently small neighborhood $U$ in $\mathcal{T}$ around a point $[t_0] \in \mathcal{T}.$  
 Here, we can apply the local Torelli theorem for $K3$ surfaces to our local period mapping  $\Phi_1$ of (\ref{localperiod}).
 Then, this theorem guarantees that
 there exists an isomorphism
 between $(S([t_1]),\psi_1)$ and $(S([t_2]),\psi_2)$
 for $[t_1],[t_2]\in U$,
 if
  $\Phi_1 ([t_1]) = \Phi_1 ([t_2]).$
 Hence, by virtue of Lemma \ref{LemmaKey},
 the local period mapping $\Phi_1$ is injective on the open set $U (\subset \mathcal{T})$.
 
If we suppose that ${\rm rank}({\rm NS} (S([t]))) $ is not equal to $16$ for generic point $[t]\in \mathcal{T}$,
we have a contradiction to the injectivity of $\Phi_1$ on $U$.
This proves the theorem.
 \end{proof}
 
 \begin{cor}\label{CorNSTr}
 For a generic point $[t] \in \mathcal{T}$,
 the lattice  ${\rm NS}(S([t] ))$ (${\rm Tr} (S([t] ))$, resp.) is given by the intersection matrix $M=U\oplus E_8 (-1) \oplus E_6 (-1)$
($A=U\oplus U \oplus A_2(-1)$, resp.).
 \end{cor}

 \begin{rem}
 Reid listed  weighted projective
$K3$ hypersurfaces with Gorenstein singularities.
They are often called `famous 95' $K3$ surfaces.
We note that our N\'eron-Severi lattice $M=U\oplus E_8 (-1) \oplus E_6 (-1)$
is equal to that of No.88 of that list (see \cite{B}).
 \end{rem}

 \section{Period mapping}

 In this section, we consider the period mapping for our $K3$ surfaces precisely.

 \subsection{Pseudo ample marked lattice polarized  $K3$ surfaces with an elliptic fibration}

 Let $S$ be a $K3$ surface.
  Let $\omega $ be the unique non-zero holomorphic $2$-form 
 on $S$ up to a constant factor.
 Regarding $\omega$ as a homomorphism
 $H_2 (S,\mathbb{Z}) \rightarrow \mathbb{C}$,
 the N\'eron-Severi lattice ${\rm NS}(S)$ is equal to its kernel.
 Letting $\rho$ be the rank of ${\rm NS}(S)$,
 then  ${\rm NS}(S)$ is a non-degenerated lattice of signature $(1,\rho-1)$. 
 We note that we can identify $H_2(S,\mathbb{Z})$ with $H^2(S,\mathbb{Z})$ 
 by the Poincar\'e duality.
 
 Also,  $H_\mathbb{R}^{1,1} (S) = H^{1,1} (S) \cap H^2 (S,\mathbb{R})$ has signature $(1,19)$.
 Let $V(S)^+$ be the component
 of 
 $V(S) = \{x\in H_\mathbb{R}^{1,1} (S)  \hspace{0.5mm} | \hspace{0.5mm} (x,x)>0 \}$
 which contains the class of a K\"ahler form on $S$.
 This is called the positive cone.
 Set
 \begin{align}
 \Delta (S) = \{\delta \in {\rm NS}(S) \hspace{0.5mm} | \hspace{0.5mm} (\delta,\delta)=-2\}.
 \end{align} 
 Let $\Delta (S)^+$ be the subset of  effective classes of $\Delta (S)$ and $\Delta(S)^- = - \Delta(S)^+$.
 Due to the Riemann-Roch theorem, we can see that
 $\Delta (S) = \Delta(S)^+ \coprod \Delta(S)^-$.
 Let $W(S)$ be the subgroup of the orthogonal group of $H_2(S,\mathbb{Z})$ generated by reflections for elements of $\Delta(S)$.
 This is called the Weyl group.
 It acts on $V(S)^+.$
 Elements of $V(S)^+$ are effective divisor classes with positive self-intersection.
 The fundamental domain for this action is given by
 $$
 C(S) =\{x\in V(S)^+ \hspace{0.5mm} | \hspace{0.5mm} (x,\delta) \geq 0 \text{ for all } \delta \in \Delta(S)^+ \}.
 $$
 We call
 $$
 C(S)^+ =\{ x\in V(S)^+ \hspace{0.5mm} | \hspace{0.5mm} (x,\delta) > 0 \text{ for all } \delta \in \Delta(S)^+ \}
 $$ 
 the K\"ahler cone for $S$.
 We set
 $$
 {\rm NS}(S)^+ = C(S) \cap H_2(S,\mathbb{Z}),  \quad \quad {\rm NS}(S)^{++} = C(S)^+ \cap H_2(S,\mathbb{Z}).
 $$
 Then,
  $x\in {NS}(S)^+$ gives a numerically effective divisor with $(x,x)>0$.
 Also, the set $ {NS}(S)^{++}$ consists of ample divisor classes.

 Let $M$ be a lattice in (\ref{latticeM}) of signature $(1,15)$,  which is embedded in the $K3$ lattice $L$.
 By fixing a generic point $[t_0]\in \mathcal{T},$
 we take a reference surface
 $S_0=S([t_0])$ as in Section 1.2.
 We obtain the corresponding $S$-marking $\psi_0: H_2(S_0,\mathbb{Z}) \rightarrow L$.
 By Corollary \ref{CorNSTr},
 we can suppose that 
 $\psi_0^{-1} (M) = {\rm NS} (S_0)$.
 Letting
 $$\Delta(M)= \{\delta \in M \hspace{0.5mm} | \hspace{0.5mm} (\delta,\delta)=-2\},$$
 we set
 $$
 \Delta(M)^+ = \{\delta \in \Delta(M) \hspace{0.5mm} | \hspace{0.5mm} \psi_0^{-1} (\delta) \in {\rm NS}(S_0) \text{ gives an effective class}\}.
 $$
 Then, letting $\Delta (M)^{-} = \{-\delta \hspace{0.5mm} | \hspace{0.5mm} \delta\in \Delta (M)^+\}$,
 we have the decomposition
 $$
 \Delta(M) = \Delta(M)^+ \coprod \Delta(M)^-.
 $$
 Set
 $$
 V(M)=\{y\in M_\mathbb{R}   \hspace{0.5mm}|\hspace{0.5mm} (y,y)>0 \}.
 $$
 Then, $V(M)$ has  two connected components.
 We can take the connected component $V(M)^+$ containing $\psi_0(x)$ for $x\in V(S_0)^+$.
 Set
 $$
 C(M)^+ = \{y\in V(M)^+ \hspace{0.5mm}|\hspace{0.5mm} (y,\delta)>0 \text{ for all } \delta \in \Delta(M)^+ \}.
 $$

 \begin{df}\label{Def:polarizedK3}
 An $M$-polarized $K3$ surface is a pair $(S,j)$ where $S$ is a $K3$ surface and $j:M\hookrightarrow {\rm NS}(S)$ is a primitive lattice embedding.
 We say that $(S,j)$ is a pseudo-ample $M$-polarized $K3$ surface if
 $$
 j(C(M)^+) \cap {\rm NS}(S)^+ \not=\phi.
 $$
 We say that $(S,j)$ is an ample $M$-polarized $K3$ surface if
 $$
 j(C(M)^+) \cap {\rm NS}(S)^{++} \not=\phi.
 $$

 \end{df}
 
 Two $M$-polarized $K3$ surfaces $(S_1,j_1) $ and $(S_2,j_2)$ are isomorphic if there exists an isomorphism of $K3$ surfaces $f: S_1\rightarrow S_2$ such that $j_2 = f_* \circ j_1$.
 
 \begin{df}
  Let $S$ be a $K3$ surface and $\psi: H_2 (S,\mathbb{Z})\rightarrow L$ be an isometry of lattices satisfying $\psi^{-1} (M) \subset {\rm NS}(S) $.
  Then, the  pair $(S,\psi)$ is called a marked $M$-polarized $K3$ surface.
   Here, the pair $(S, \psi^{-1}\hspace{0.5mm} | \hspace{0.5mm}_M)$ is a $M$-polarized $K3$ surface in the sense of Definition \ref{Def:polarizedK3}.
  If  $(S, \psi^{-1}\hspace{0.5mm} | \hspace{0.5mm}_M)$ is a pseudo-ample $M$-polarized $K3$ surface,
  then we call $(S,\psi)$ a pseudo-ample marked $M$-polarized $K3$ surface.
 \end{df}
 
 \begin{df}
 Suppose $(S_1,\psi_1)$ and $(S_2,\psi_2)$ are two marked pseudo-ample $M$-polarized $K3$ surfaces.
 If $(S_1,\psi_1^{-1}\hspace{0.5mm} | \hspace{0.5mm}_M) $ is isomorphic to $(S_2,\psi_2^{-1}\hspace{0.5mm} | \hspace{0.5mm}_M) $
 as $M$-polarized $K3$ surfaces,
 we say that $(S_1,\psi_1)$ and $(S_2,\psi_2)$ are isomorphic as pseudo-ample $M$-polarized $K3$ surfaces.
 Moreover, if there exists an isomorphism 
 $f:S_1\rightarrow S_2$ such that
 $\psi_1 = \psi_2 \circ f_*$,
 then 
 we say that 
 $(S_1,\psi_1)$ and $(S_2,\psi_2)$ are isomorphic as marked pseudo-ample $M$-polarized $K3$ surfaces.
 \end{df}

Let $\mathcal{M}_M$ be the fine moduli space of marked $M$-polarized $K3$ surfaces.
 This can be obtained by gluing local moduli space of marked $M$-polarized $K3$ surfaces.
 The local period mappings are glued together to give a holomorphic mapping
 \begin{align}\label{per}
 per: \mathcal{M}_M \rightarrow \mathcal{D}_M.
 \end{align}
 Here, $\mathcal{D}_M$ is the symmetric space in (\ref{DM}).
 The mapping (\ref{per}) is \'etale and surjective.

 \begin{thm}(Dolgachev \cite{D} Section 3) \label{ThmDolg}
 The restriction of the period mapping (\ref{per})
 to the subset $\mathcal{M}_M^{pa}$ of isomorphism classes of marked pseudo-ample $M$-polarized $K3$ surfaces gives a surjective mapping
 \begin{align}\label{PA-Period}
 per' : \mathcal{M}_M^{pa} \rightarrow \mathcal{D}_M.
 \end{align}
 \end{thm}

 \begin{rem}
 Let $\mathcal{M}_M^a$ be the set of isomorphism classes of marked  ample isomorphism classes.
 Set $\mathcal{D}_M^\circ = \mathcal{D}_M-\bigcup_{\delta \in \Delta(A)} (H_\delta \cap \mathcal{D}_M)$, where
 $\Delta(A)=\{\delta\in A \hspace{0.5mm} | \hspace{0.5mm} (\delta,\delta)=-2\}$. 
 Then,
 the period mapping (\ref{PA-Period}) induces a bijection between $\mathcal{M}_M^a$ and $\mathcal{D}_M^\circ$ (see \cite{D}).
 \end{rem}

In this paper,
 the orthogonal group for the lattice $M$ is denoted  by $O(M)$.
Let us consider the group
$$
\Gamma(M) = \{\sigma \in O(L) \hspace{0.5mm} | \hspace{0.5mm} \sigma(m) = m \text{ for any } m \in M\},
$$
which acts on the moduli space $\mathcal{M}_M$
by
$(S,\psi)\mapsto (S,\psi\circ \sigma)$.
This action does not change the isomorphism class of the $M$-polarized $K3$ surface $(S,\psi^{-1}\hspace{0.5mm} | \hspace{0.5mm}_M)$. 
 So,
$ \mathcal{M}_M^{pa} / \Gamma (M)$
gives the isomorphism classes of pseudo ample $M$-polarized $K3$ surfaces.
 
 Here,
 note that 
 there exists an injective homomorphism
 $\Gamma(M) \rightarrow O(A)$.
 Let $\Gamma_M$ be the image of this injection.
 We can see that
 $\Gamma _M$
 is equal to
 \begin{align}\label{StableO}
  \tilde{O}(A) = {\rm Ker}(O(A) \rightarrow {\rm Aut}(A^\vee/A)).
 \end{align}
 Here, $A^\vee ={\rm Hom}(A,\mathbb{Z})$,
 $A^\vee / A$ is the discriminant group of $A$
and $O(A) \rightarrow {\rm Aut}(A^\vee/A)$ is the natural homomorphism.
The group $\tilde{O}(A)$ is called the stable orthogonal group of $A$.
 We remark that $\tilde{O}(A)$ is a subgroup of finite index in $O(A)$.

 \begin{thm}(Dolgachev \cite{D} Section 3) \label{ThmDolg}
 The period mapping (\ref{PA-Period}) gives the following bijection
 \begin{align}\label{PerIso}
 \mathcal{M}_M^{pa} / \Gamma (M) \simeq \mathcal{D}_M / \tilde{O}(A).
 \end{align}
 Namely, the quotient space $\mathcal{D}_M / \tilde{O}(A)$ gives the set of isomorphism classes of pseudo-ample $M$-polarized $K3$ surfaces.
 \end{thm}

 \begin{rem}
 Due to \cite{GHSz} Lemma 3.2,
 the index  $[O(A):\tilde{O}(A)]$ is equal to the order of the finite group $O(q_A)$.
 Here, $q_A$ is the discriminant quadratic form.
 In our case of the lattice $A$ of (\ref{latticeA}),
 we can see that $O(q_A)\simeq \mathbb{Z}/2\mathbb{Z}$. 
 Hence,  the stable orthogonal group $\Gamma_M=\tilde{O}(A)$ does not coincide with $O(A)$. 
 \end{rem}

   \subsection{$P$-marking and period mapping}

 Let us recall some basic properties of $K3$ surfaces
 (for detail, see \cite{LP} or \cite{L}).
 Letting $S$ be a $K3$ surface,
 if $x\in {\rm NS}(S)$ satisfies $(x,x) \geq -2$,
 then one of $x$ or $-x$ is representable by an effective divisor.
  An irreducible curve $C$ is called a nodal curve if $(C,C)=-2$.
 We can see that a nodal curve $C$ is smooth and rational.
  Let $B(S)\subset \Delta(S)$ be the set of classes of nodal curves.
The set $B(S)$ gives a root basis for $\Delta(S)$.
 This implies that 
 \begin{align}\label{RootIneq}
 (x,y) \geq 0 \quad \quad (\text{for all distinct } x,y\in B(S)).
 \end{align}
 If a primitive element $D\in {\rm NS}(S) -\{0\}$ is on the boundary of the positive cone and $(D,\alpha)\geq 0$ for all $\alpha \in B(S)$,
 then the linear system $ | D | $ defines an elliptic fibration.

 Now, let us introduce  $P$-marked $K3$ surfaces
to study the period mapping for our elliptic $K3$ surface $S([t])$ of (\ref{S(t)}) precisely. 
 Take a reference surface $S_0 = S([t_0])$ for a fixed $[t_0]\in \mathcal{T}$. 
 Take an $S$-marking $\psi_0:H_2(S_0,\mathbb{Z}) \rightarrow L$
 as in Section 1.2. 
 Let $S $ be an algebraic $K3$ surface.
 An isometry $\psi : H_2 (S,\mathbb{Z}) \rightarrow L$ is called a $P$-marking of $S$ 
 if it satisfies the  following conditions:
 \begin{itemize}
 \item[(i)] $\psi^{-1} (M) \subset {\rm NS} (S)$ for the lattice $M$ of (\ref{latticeM}),
  
 \item[(ii)]  For $a_j,b_j,O,F$ in (\ref{generators}), $\psi^{-1} \circ \psi_0 (a_j), \psi^{-1}\circ \psi_0 (b_j), \psi^{-1}\circ \psi _0(O)$ and $\psi^{-1}\circ \psi_0(F)$
 give effective divisors on $S$, 
 
 \item[(iii)] $\psi^{-1}\circ \psi_0 (F)$ is a nef divisor. 
 
 \end{itemize}
 The pair $(S,\psi)$ is called a $P$-marked $K3$ surface.
  Also, we define the equivalent classes and isomorphic classes
 as in Definition \ref{DfS-marking}.
 
 We can define the period 
\begin{align}\label{P-Period}
\Phi ((S,\psi)) = \Big( \int_{\psi^{-1} \circ \psi_0 (\gamma_1)} \omega : \cdots : \int_{\psi^{-1} \circ \psi_0 (\gamma_6)} \omega \Big)
\end{align}
 of a $P$-marked $K3$ surface $(S,\psi)$,
 where $\omega$ is the unique holomorphic $2$-form on $S$ up to a constant factor.

 \begin{lem}\label{LemP-S(t)}
 Any equivalent classes of $P$-marked $K3$ surfaces $(S,\psi)$ are given by  explicit elliptic $K3$ surfaces $S([t])$ $([t] \in T)$ defined by the equation (\ref{S(t)}).
 \end{lem}
 
 \begin{proof}
 We note that
 a divisor $D$ is nef if and only if
 $(D, C)\geq 0$ for any effective divisor $C$ on $S$.
 By the Riemann-Roch Theorem,
 we can see that 
 there exists a unique elliptic fibration $\pi:S \rightarrow \mathbb{P}^1(\mathbb{C})=x\text{-sphere}$
 such that
 $\psi^{-1} \circ \psi_0(F)$ is a general fibre of $\pi$
 and 
 $\psi^{-1} \circ \psi_0(O)$ is the zero-section of $\pi$.
 By considering the intersections of  $\psi^{-1} \circ \psi_0(F)$, $\psi^{-1} \circ \psi_0 (a_j)$ and $\psi^{-1} \circ \psi_0 (b_j)$,
 we can suppose  that  $\psi^{-1}\circ \psi_0(a_j)$  ($\psi^{-1} \circ \psi_0(b_j)$, resp.) is a component of the singular fibre $\pi^{-1}(\infty)$ ($\pi^{-1}(0)$, resp.).
 It follows that $\pi^{-1}(\infty)$ ($\pi^{-1}(0)$, resp.) should contain the subgraph of Kodaira type $II^*$ ($IV^*$, resp.). 
 According to the classification of singular fibres of elliptic fibration,
  $\pi^{-1}(\infty)$ ($\pi^{-1}(0)$, resp.) should be of type $II^*$ ($IV^*,  III^*$ or $II^*$, resp.).
Then, by considering the Weierstrass models of 
elliptic surfaces,
we can see that 
such an elliptic $K3$ surface $(S,\pi,\mathbb{P}^1(\mathbb{C}))$ can be  given by $S([t])$ $([t]\in T)$ of (\ref{S(t)}).
 
 Moreover, 
 due to Lemma \ref{Lemma1}
 and 
 a similar argument to the proof of Lemma \ref{LemmaKey},
 we can prove that
 $\{S([t]) \hspace{0.5mm} | \hspace{0.5mm} [t]\in T \}$ gives
 the set of equivalent classes of $P$-marked $K3$ surfaces.
 \end{proof}

 \begin{lem}\label{PropNodal}
 Let $(S_0,\psi_0)$ be an $S$-marked $K3$ surface.
Let $(S,\psi)$ be a marked pseudo-ample $M$-polarized $K3$ surface.
If $C\in {\rm NS}(S_0)$ be a nodal curve,
then $\psi^{-1} \circ \psi_0(C) \in {\rm NS} (S)$ is a nodal curve. 
\end{lem} 

\begin{proof}
Since
$$
(\psi^{-1}\circ \psi_0 (C) , \psi^{-1}\circ \psi_0 (C) ) =(C,C) =-2,
$$
one of
$\psi^{-1}\circ \psi_0 (C)$ or $-\psi^{-1} \circ \psi_0 (C)$ 
is effective.
From the definition,
 we can see that 
there exists $\kappa \in C(M)^+$ such that $\psi^{-1} (\kappa) $ is in the closure of the K\"ahler cone.
Due to the argument in Section 2.1,
$\psi_0^{-1} (\kappa) \in H_2(S_0,\mathbb{Z})$ gives an ample divisor class.
According to Nakai's criterion, we have
$$
(\psi^{-1} (\kappa), \psi^{-1} \circ \psi_0 (C)) = (\psi_0^{-1} (\kappa) ,C) >0.
$$
So, it follows that $\psi^{-1} \circ \psi_0(C)$ is effective.
\end{proof}

 \begin{thm}\label{ThmMPAP}
Let $(S,\psi)$ be a marked pseudo-ample $M$-polarized $K3$ surface.
Then, $\psi: H_2(S,\mathbb{Z}) \rightarrow L$ gives a $P$-marking for $S$.
 \end{thm}

\begin{proof} 
Since $a_j$ $(j\in \{0,\cdots,8\})$ is  a nodal curve on $S_0$,
due to Lemma \ref{PropNodal},
$\psi^{-1} \circ \psi_0 (a_j)$
is a nodal curve on $S$.
  Take an arbitrary nodal curve $C$ on $S$.
   If $C$ is linearly equivalent to $\psi^{-1} \circ \psi_0 (a_j)$ for some $j\in\{0,\cdots,8\}$,
 $C$ is a nodal curve in an elliptic fibre.
 So, we have  $
 (\psi^{-1}\circ \psi_0 (F),C)= 0.
 $
 If $C$ is not linearly equivalent to $\psi^{-1} \circ \psi_0 (a_j)$ $(j\in \{0,\cdots,8\})$,
 then 
 \begin{align}\label{CNef}
 (\psi^{-1} \circ \psi_0 (a_j), C) \geq 0,
 \end{align} 
 since $B(S)$ gives a root system (recall (\ref{RootIneq})).
 Due to (\ref{F-linear}),
 we have
 \begin{align}\label{psiF}
\notag
 \psi^{-1} \circ \psi_0 (F) = 
 &
 \psi^{-1} \circ \psi_0( a_0) + 2 \psi^{-1} \circ \psi_0(a_1) + 3 \psi^{-1} \circ \psi_0(a_2) + 4 \psi^{-1} \circ \psi_0(a_3)\\
 & + 5\psi^{-1} \circ \psi_0( a_4) + 6\psi^{-1} \circ \psi_0( a_5) + 3 \psi^{-1} \circ \psi_0(a_6) + 4 \psi^{-1} \circ \psi_0(a_7) + 2 \psi^{-1} \circ \psi_0(a_8)
 \end{align}
 in ${\rm NS}(S)$.
 According to (\ref{CNef}) and (\ref{psiF}),
 we have
 $
 (\psi^{-1}\circ \psi_0 (F),C)\geq 0.
 $
Therefore, 
$\psi^{-1}\circ \psi_0 (F) $ has a non-negative intersection number with
for any nodal curve on $C$. 
 This means that $\psi^{-1}\circ \psi_0 (F) $ is nef.
 \end{proof}

The orthogonal group $O(A)$ for the lattice $A$ acts on the space $\mathcal{D}_M$.
 Let $\mathcal{D}$ be a connected component of $\mathcal{D}_M$.
 Set 
 $$
 O^+(A) = \{\gamma \in O(A) \hspace{0.5mm} | \hspace{0.5mm} \gamma(\mathcal{D}) = \mathcal{D}\}.
 $$
 We set
 \begin{align}\label{GammaA}
 \Gamma = \tilde{O}^+ (A) = \tilde{O}(A) \cap O^+(A),
 \end{align}
 where  $\tilde{O}(A)$ is the stable orthogonal group of (\ref{StableO}).
 We note that $[\tilde{O}(A) : \tilde{O}^+ (A)]=2.$
 Especially, we have
 $$
 \mathcal{D}_M / \tilde{O}(A) = \mathcal{D}/\tilde{O}^+(A).
 $$
 Due to the above theorem, Theorem \ref{ThmDolg} and the definitions of period mappings,
 we have the following corollary.

 \begin{cor}\label{CorPer}
 The period mapping (\ref{P-Period}) is explicitly given by
 \begin{align}\label{PerPhi}
 \Phi: T\ni [t] \mapsto \Big(\int_{\psi^{-1} \circ \psi_0 (\gamma_1)} \omega_{[t]} : \cdots: \int_{\psi^{-1} \circ \psi_0 (\gamma_6)} \omega_{[t]} \Big) \in \mathcal{D}.
 \end{align}
 This gives an extension of (\ref{S-Per}) defined on $\mathcal{T}$.
Also, for the group $\Gamma$ of (\ref{GammaA}),
the mapping (\ref{PerPhi}) induces the isomorphism 
\begin{align}\label{PhiIso}
\Phi: T \simeq \mathcal{D}/\Gamma=Q.
\end{align}
 \end{cor}

 \subsection{Family $S_{CD} (\alpha,\beta,\gamma,\delta)$ }

 In \cite{CD}, the family
 \begin{align}\label{SCD}
 S_{CD}  (\alpha,\beta,\gamma,\delta): z_1^2 =y_1^3+ (-3\alpha x_1^4-\gamma x_1^5) y_1+ (x_1^5-2\beta x_1^6 +\delta x^7_1) 
 \end{align}
 of $K3$ surfaces was studied.
 This family is closely related to  famous Siegel modular forms of degree $2$.

First, we recall the following result.
 
 \begin{prop}\label{PropCD} (\cite{CD}, \cite{Kumar} or \cite{NS})
 For a generic point $(\alpha: \beta: \gamma: \delta)\in \mathbb{P}(4,6,10,12)$,
 the N\'eron-Severi lattice ${\rm NS} (S_{CD} (\alpha,\beta,\gamma,\delta))$ (transcendental lattice ${\rm Tr} (S_{CD} (\alpha,\beta,\gamma,\delta)) $, resp.) of the $K3$ surface $S_{CD} (\alpha,\beta,\gamma,\delta)$
 is 
 given by the intersection matrix
 $U\oplus E_8 (-1) \oplus E_7 (-1)$
 ($U\oplus U \oplus A_1(-1)$, resp.).
 \end{prop}

 Let us recall the meaning and the importance
 of this family.
 By taking the subspace 
 $$
 T_0 = \{(\alpha:\beta :\gamma :\delta)\in \mathbb{P}(4,6,10,12) \hspace{0.5mm} | \hspace{0.5mm} \gamma=\delta=0 \},
 $$ 
 we can obtain the period mapping
 \begin{align} \label{PerSiegel}
T_0 \ni (\alpha:\beta:\gamma:\delta) \mapsto [\xi_0] \in \mathfrak{S}_2/Sp(4,\mathbb{Z}),
 \end{align}
 where $\mathfrak{S}_2$ is the Siegel upper half plane of degree $2$ and $Sp(4,\mathbb{Z})$ is the symplectic group.
 By the inverse period mapping (\ref{PerSiegel}),
 $\alpha$ ($\beta, \gamma,\delta$, resp.)
 gives a Siegel modular form of weight $4$ ($6$, $10$, $12$, resp.).
More precisely,
we have the following result.

\begin{prop}(\cite{I35}, \cite{I},   \cite{CD}, \cite{Kumar} and \cite{NS}) \label{PropSiegelModular}
(1) 
The ring of Siegel modular forms on $\mathfrak{S}_2$ for the group $Sp(4,\mathbb{Z})$ of even weight and the trivial character   is generated  by modular forms $\alpha,\beta,\gamma$ and $\delta$ of weight $4,6,10$ and $12$, respectively.
Also, there exists a modular form $\chi_5$ ($\chi_{30}$, resp.) of weight $5$ ($30$, resp.) and  non-trivial character.
Moreover,
setting
$\chi_{35}=\chi_{5} \chi_{30}$,
the space of Siegel modular forms on $\mathfrak{S}_2$ for the group $Sp(4,\mathbb{Z})$ of odd weight and the trivial character is equal to $\chi_{35}\mathbb{C}[\alpha,\beta,\gamma,\delta]$.

(2) More precisely,
$\alpha, \beta, \gamma$ and $\delta$ have explicit expressions by the Igusa invariants:
\begin{align}\label{alphaIgusa}
\begin{cases}
&\alpha = \frac{1}{9} I_4, \quad \beta=\frac{1}{27}(-I_2 I_4 +3 I_6),  \\
&\gamma=8 I_{10}, \quad \delta=\frac{2}{3}I_{2}I_{10}.
\end{cases}
\end{align}
\end{prop}

 There exist relations which determine the structure of the ring of Siegel modular forms of degree $2$.
 Namely, 
 $\chi_5^2$ 
 and $\chi_{30}^2$  
 are given by the polynomials in $\alpha,\beta,\gamma$ and $\delta$.
 These relations are calculated by the explicit defining equation (\ref{SCD}) of $K3$ surfaces as follows.
 Let 
 $$R_0 (x_1)=
  \frac{1}{x_1^{10}}(\text{ the discriminant of the right hand side of (\ref{SCD}) in } y_1).
  $$
 Then, the discriminant of the polynomial $R_0(x_1)$ in $x_1$
 is given by the product of 
 $\gamma^3$,
 $r_0^3$ and $d_0$ up to a constant factor.
  Here, $r_0$ is equal to the resultant  of  
  $\frac{g_2(x_1)}{x_1^4}$ and that of $\frac{g_3(x_1)}{x_1^5}$,
 where we express (\ref{SCD}) as $z_1^2 = y_1^3 - g_2(x_1) y_1 - g_3(x_1)$.
 Also, $d_0$ is a polynomial of weight $60$ in $\alpha,\beta,\gamma$ and $\delta$. 
  The factor $\gamma^3$ implies that the square of $\chi_5$ is equal to $\gamma$ up to a constant factor.
 Also,
 via (\ref{alphaIgusa}),
 we can see that $d_0$  is essentially equal to a factor of weight $60$ 
  in the famous relation for $\chi_{30}^2$
  by Igusa (see \cite{I35} p.849)
 up to a constant factor.

 Thus,  the famous Siegel modular forms of degree $2$ 
 can be calculated by
 inverse period mappings for this family of $K3$ surfaces (\ref{SCD}).

 Our  $M$-polarized $K3 $ surfaces $S([t])$ $([t]\in T)$ is  
 closely related to the Clingher-Doran family as follows.

 \begin{prop}\label{PropS(t)CD}
 The family  $\{S_{CD} (\alpha,\beta,\gamma,\delta) \hspace{0.5mm} | \hspace{0.5mm} (\alpha:\beta: \gamma: \delta) \in T_0 \}$ can be regarded as a subfamily  of $\{S([t]) \hspace{0.5mm} | \hspace{0.5mm} [t] \in T \}$ corresponding to $t_{18}=0.$
 \end{prop}

 \begin{proof}
 By putting $x_1=\frac{1}{x_0},y_1=\frac{y_0}{x_0^4}$ and $z_1=\frac{z_0}{x_0^6}$, 
 (\ref{SCD}) is transformed to
 \begin{align}\label{align}
 z_0^2=y_0^3 + (-3 \alpha  x_0^4-\gamma x_0^3)y_0 + (x_0^7 -2\beta x_0^6 + \delta x_0^5).
 \end{align}
 This is equal to the form which is coming from (\ref{S(t)}) by the substitution
 $x=\frac{x_0}{w^6},y=\frac{y_0}{w^{14}},z=\frac{z_0}{w^{21}}, t_4=-3\alpha,  t_6=-2\beta, t_{10}=-\gamma, t_{12}=\delta$ and $t_{18}=0.$  
 \end{proof}

 We shall use Proposition \ref{PropS(t)CD}
 to prove the main theorem at end of this paper.

 \section{Lattice with the Kneser conditions}

 \subsection{The Kneser conditions}
 
 In this section,
 we survey arithmetic properties of lattices.
 They are useful to study our period mapping of $K3$ surfaces precisely.
 
 Let $\mathcal{K}$ be an even integral lattice.
 Let $\mathcal{K}^\vee$ be its dual lattice.
 As in Section 2, letting $O(\mathcal{K})$ be the orthogonal group for the lattice $\mathcal{K}$,
 we define the stable orthogonal group
$
 \tilde{O} (\mathcal{K}) = {\rm Ker} (O(\mathcal{K}) \rightarrow O(\mathcal{K}^\vee /\mathcal{K}))$.
We can define the spinor norm $sn_\mathbb{R}: O(\mathcal{K}\otimes\mathbb{R}) \rightarrow \{\pm 1 \}$.
We set $O^+(\mathcal{K}) = O(\mathcal{K}) \cap {\rm Ker}(sn_\mathbb{R})$
 and
 \begin{align}\label{Gamma}
 \Gamma = \tilde{O}^+ (\mathcal{K}) = \tilde{O} (\mathcal{K}) \cap O^+(\mathcal{K}). 
 \end{align}

 \begin{df}
 If an even integral lattice $\mathcal{K}$ satisfies the following conditions,
 we say that $\mathcal{K}$ satisfies the Kneser conditions:
 \begin{itemize}
 \item[(i)] For the signature $(s_+,s_-)$ of $\mathcal{K} $, ${\rm min}(s_+,s_-)\geq 2$,
  
 \item[(ii)] There exists $x\in \mathcal{K}$ such that $(x, x) =-2$,
 
 \item[(iii)] ${\rm rank}_2 (\mathcal{K}) \geq 6$,
 
 \item[(iv)] ${\rm rank}_3 (\mathcal{K}) \geq 5$.
 \end{itemize}
Here, ${\rm rank}_p (\mathcal{K})$ for a prime number $p$ means the rank of the mod $p$ reduction of $\mathcal{K}$. 
 \end{df}
 
 Such lattices were firstly studied by Kneser \cite{Kneser}.
We will use the following result due to Gritsenko-Hulek-Sankaran. 
 
 \begin{thm}(\cite{GHS}) \label{ThmKneser}
 Suppose $\mathcal{K}$ satisfies the Kneser conditions and contains at least two hyperbolic planes $U$.
 
(1) The group $\Gamma$ of (\ref{Gamma}) is generated by reflections
$
 \sigma_\delta: z\mapsto z+ (z,\delta) \delta
 $
  $(\delta\in \mathcal{K}$, $(\delta, \delta)=-2)$.
 
 (2) 
 $
 {\rm Char} (\Gamma) =\{{\rm id}, {\rm det}\}.
 $
 \end{thm}

 In order to apply this theorem to period mappings for $K3$ surfaces, we will consider cases that the signature of $\mathcal{K}$ is of type $(2,n)$.
 We note that
 the lattice $A$ of (\ref{latticeA}) satisfies the conditions of the above theorem.
 In fact,
 our lattice $A$ is the simplest lattice among such lattices.

 \subsection{Comparison with the results of \cite{CD},  \cite{MSY} and \cite{CMS}}

  In this subsection,
  let us see features of our family of $K3$ surfaces
  by comparing other famous families of  polarized $K3$ surfaces.

  First,
  let us recall the family of $K3$ surfaces in Section 2.3 due to \cite{CD}.
  The lattice
  $$A_S = U \oplus U \oplus A_1(-1).$$
  does not satisfy the Kneser conditions.
 Nevertheless, this lattice also has good properties.
 We can see that
  $
  O(A_S)
  $
  is equal to
  the stable orthogonal group
  $\tilde{O}(A_S)$.
  According to Proposition \ref{PropSiegelModular},
  the classical Siegel modular forms of degree $2$ are given by the modular forms for the lattice $O(A_S)$.
  Moreover,
  in \cite{GN},
  it  is proved that
   the group $O(A_S) $ is generated by 
  reflections for vectors with self-intersection $-2$
  and
  its character group
  is isomorphic to $\mathbb{Z}/2\mathbb{Z}$.
  Thus,
    the lattice $A_S$
    has
    good arithmetic properties.

  Let us see another  example in a famous work due to Matsumoto-Sasaki-Yoshida \cite{MSY}.
  They studied a family of  $K3$ surfaces with four complex parameters.
  From the viewpoint of  hyperplane arrangements and differential equations,
  their family is very natural and interesting.
  Their $K3$ surfaces are given by the double covering
  of $\mathbb{P}^2(\mathbb{C})$
  blanched along six lines.
  The configuration of these lines is coming from the Grassmannian manifold $Gr(3,6)$.
 The periods for their $K3$ surfaces
 satisfy a system of partial differential equations,
 which is called a hypergeometric equation of type $(3,6)$.
 This system gives a natural extension of the classical Gauss hypergeometric equation.
Here, a generic member of their family is a lattice polarized $K3$ surface with the transcendental lattice
 $$A_{MSY} = U(2) \oplus U(2) \oplus (-I_2 (2))$$
of signature $(2,4)$. 
We note that
this lattice does not satisfy the Kneser conditions.
We note that Matsumoto \cite{Matsumoto} studied modular forms
 for the group $\Gamma_{A_{MSY}}$,
 which is given by the lattice $A_{MSY}$,
  coming from the inverse period mapping of a family of $K3$ surfaces with four parameters.
 However,
 the modular group $\Gamma_{A_{MSY}}$
 has more complicated structure
 than our modular group.
Namely,  it is generated by
 reflections  for vectors with self-intersection not only $-2$ but also  $-4$.

 The family of \cite{MSY} was
 studied more thoroughly
 in
   a recent work of Clingher-Malmendier-Shaska \cite{CMS} as follows.
   They explicitly determined
  the family of the van Geemen-Sarti partners of the $K3$ surfaces of \cite{MSY}.
 Namely, a generic member of \cite{CMS} has a special Nikulin involution which induces a double covering of a generic member of \cite{MSY}.
Here, a generic member of this family  is a lattice polarized $K3$ surface  with the transcendental lattice
 $$
 A_{CMS} = U \oplus U\oplus (- I_2(2)).
 $$
 The lattice $A_{CMS}$ does not satisfy the Kneser conditions also.
  The modular group $\Gamma_{A_{CMS}}$ corresponding to the lattice $A_{CMS}$ is isomorphic to a finite extension of $\Gamma_{A_{MSY}}$.
Also, we note that this family can be applied to string theory (see \cite{CMS} and \cite{CHM}).

  From the viewpoint of Dynkin diagrams and elliptic $K3$ surfaces,
our family corresponds to $E_8$ and $E_6$ as in Section 1.2,
while the family of \cite{CMS} corresponds to $E_7$ and $E_7$.
Each our family and the family of \cite{CMS} respectively gives a natural extension of the family of \cite{CD},
which  corresponds to $E_8$ and $E_7$.

 The purpose of this paper
 is to obtain modular forms for our modular group $\Gamma = \tilde{O}^+(A)$ of (\ref{GammaA})
 by the inverse period mapping of our $K3$ surfaces.
Our modular forms are defined on a $4$-dimensional bounded  symmetric domain of type $IV$, which is the homogeneous space for the Lie group $SO_0(2,4) \simeq SU(2,2)$.
Also, the modular forms coming from the families \cite{MSY} and \cite{CMS}
are also defined on the symmetric domain corresponding to $SU(2,2)$.
 However, 
 our modular forms are different from their modular forms,
 since our modular group $\Gamma$ is different from $\Gamma_{A_{MSY}}$ and $\Gamma_{A_{CMS}}$.
Each result
has different meanings and features respectively. 
For example,
it is one of the features of our result  that 
our lattice $A$ and our modular group $\Gamma$ have nice arithmetic properties as in Theorem \ref{ThmKneser}.
 In the following contents, 
 in order to study our modular forms,
we  will apply  techniques of canonical orbibundles over  orbifolds derived from the action of $\Gamma$ on $\mathcal{D}$.  
Such techniques are based on the nice properties of $\Gamma$.

 \subsection{Structure of branches}

 We can apply 
  Theorem \ref{ThmKneser}
  to our lattice $A$.
 So, the group $\Gamma = \tilde{O}^+ (A)$ is generated by the reflections $\sigma_\delta$ such that $(\delta,\delta)=-2$.
 Let $\Delta(A)$ be the set of vectors in $A$ with self-intersection $-2$.

\begin{lem} \label{LemGammaProj}
One has $O^+(A)/\Gamma \simeq \mathbb{Z}/2\mathbb{Z}$ and $-{\rm id}_{O^+(A)} \not\in \Gamma$.
Especially, $\Gamma$ is isomorphic to the projective group $PO^+(A)$.
\end{lem} 
 
 \begin{proof}
 Let $\{v_1,v_2\}$ be a system of basis of $A_2(-1) (\subset A)$ such that $(v_1\cdot v_1)=(v_2\cdot v_2)=-2$ and $(v_1\cdot v_2)=1.$
 Setting $y_1=\frac{1}{3} v_1 + \frac{2}{3} v_2 \mod A$ and $y_2=\frac{2}{3}v_1 +\frac{1}{3} v_2 \mod A$,
 we obtain $A^\vee / A =\{0,y_1,y_2\}$.
 So, by the definition (\ref{GammaA}) of $\Gamma$, we can directly check that $O^+(A)/\Gamma \simeq \mathbb{Z}/2\mathbb{Z}$ and $-{\rm id}_{O^+(A)} \not\in \Gamma$.
 Also, we can see that $\Gamma \simeq PO^+(A)$.
 \end{proof}

By the above lemma,
 the action of $\Gamma$ on $\mathcal{D} $ is effective, since that of $PO^+(A)$ on $\mathcal{D}$ is so.
 Now,
 let us consider the set  $\left\{[Z] \in \mathcal{D} \hspace{0.5mm} | \hspace{0.5mm} g ([Z]) = [Z] \right\}$
 of fixed points
 for  $g\in \Gamma$
 and the union
 \begin{align*}
\mathfrak{H}_{\mathcal{D}} = \bigcup_{g \in \Gamma} \left\{ [Z] \in \mathcal{D} \hspace{0.5mm} |\hspace{0.5mm} g([Z]) =  [Z] \right\}.
\end{align*}
 Also, letting
 $\Gamma_{[Z]}=\left\{g\in \Gamma \hspace{0.5mm}|\hspace{0.5mm} g([Z])= [Z] \right\}$
 be the stabilizer subgroup  with respect to $[Z]\in \mathcal{D} $,
 we set
 \begin{align*}
 \mathfrak{S}_{\mathcal{D}} =\left\{[Z]\in \mathcal{D} \hspace{0.5mm} | \hspace{0.5mm} \Gamma_{[Z]} \text{ is neither } \{{\rm id_\Gamma}\} \text{ nor } \{{\rm id_\Gamma}, \sigma_\delta \}\simeq \mathbb{Z}/2\mathbb{Z} \text{ for } \delta \in \Delta (A) \right\}.
 \end{align*}
 Then,
 $\mathfrak{H}_{\mathcal{D}}$ ($\mathfrak{S}_{\mathcal{D}}$, resp.) is a countable union of hypersurfaces
 (an analytic subset of codimension at least $2$, resp.),
 since $\Gamma$ satisfies Theorem \ref{ThmKneser} and $\mathcal{D}$ is a subset of projective space.

 Any points $[Z]$ such that $\Gamma_{[Z]} \not=\{{\rm id_\Gamma}\}$   are contained in 
  $\mathfrak{H}_{\mathcal{D}} \cup \mathfrak{S}_{\mathcal{D}}$.
Therefore,
 the action of $\Gamma$ on $\mathcal{D} - (\mathfrak{H}_{\mathcal{D}} \cup \mathfrak{S}_{\mathcal{D}})$ is free.
 Also,
  $\Gamma_{[Z]}$ for $[Z]  \in \mathfrak{H}_{\mathcal{D}}- (\mathfrak{H}_{\mathcal{D}} \cap \mathfrak{S}_{\mathcal{D}})$
 is given by
 $\{{\rm id}, \sigma_\delta \}$
 for some $\delta \in \Delta (A).$
 Let  $\mathfrak{H}_{Q}$ and $\mathfrak{S}_{Q}$ be the subsets in $Q=\mathcal{D}/\Gamma$
 given by the images of $\mathfrak{H}_\mathcal{D}$ and $\mathfrak{S}_\mathcal{D}$, respectively.

 Set
 $\Gamma' = \{\gamma \in \Gamma \hspace{0.5mm} | \hspace{0.5mm} {\rm det} (\gamma) =1\}$.
 This is a normal subgroup of $\Gamma$.
 From Theorem \ref{ThmKneser}, 
 $$
 \Gamma'=\{ \gamma \in \Gamma \hspace{0.5mm} | \hspace{0.5mm} \gamma \text{ can be given by a product of reflections of even numbers }  \}.
 $$
 Set
 $Q_1 = \mathcal{D}/\Gamma'$.
 We naturally define the subsets
 $\mathfrak{H}_{Q_1}$ and $\mathfrak{S}_{Q_1}$ in $ Q_1.$ 
We note that 
$Q_1-\mathfrak{S}_{Q_1} \rightarrow Q-\mathfrak{S}_Q$
is a double covering 
branched along
$\mathfrak{H}_Q-\mathfrak{S}_Q.$ 
 Since the action of $\Gamma$ on $\mathcal{D}-(\mathfrak{H}_\mathcal{D} \cup \mathfrak{S}_\mathcal{D})$
 is free,
 any points whose stabilizer group is $\mathbb{Z}/2 \mathbb{Z}$
 are contained in $\mathfrak{H}_\mathcal{D}$.
Moreover,  $\mathfrak{H}_\mathcal{D}$ coincides with the set of fixed points of $\Gamma/\Gamma' \simeq \mathbb{Z}/2\mathbb{Z}$.
Hence, it follows that the action of $\Gamma'$ on $\mathcal{D}-\mathfrak{S}_\mathcal{D}$ is free.
 
 Here, let us recall the period mapping $\Phi:T \simeq Q$ of (\ref{PhiIso}), which gives an isomorphism.
 We set $\mathfrak{H}_T=\Phi^{-1} (\mathfrak{H}_Q)$ and $ \mathfrak{S}_T =\Phi^{-1}(\mathfrak{S}_Q).$
 Since $T$ is a Zariski open set of the weighted projective space $\hat{T}=\mathbb{P}(4,6,10,12,18)$
 and $\mathfrak{H}_T$ is an analytic subgroup
 of codimension $1$,
 there exists a weighted homogeneous polynomial 
 $\Delta_T(t) \in \mathbb{C}[t_4,t_6,t_{10},t_{12},t_{18}]$
 such that
 \begin{align}\label{DeltaT}
 \mathfrak{H}_T = \{[t]= (t_4:t_6:t_{10}:t_{12}:t_{18}) \in \mathbb{P}(4,6,10,12,18) \hspace{0.5mm} | \hspace{0.5mm} \Delta_T (t)=0 \}.
 \end{align}
 Let us consider the double covering $T_1$ of $T-\mathfrak{S}_T$ branched along $\mathfrak{H}_T-\mathfrak{S}_T$:
 \begin{align}\label{T1}
 T_1=\{([t],s) \in (T-\mathfrak{S}_T) \times \mathbb{C} \hspace{0.5mm} | \hspace{0.5mm} s^2 =\Delta_T(t) \}.
 \end{align}
 
 Since $T_1$ ($Q_1-\mathfrak{S}_{Q_1}$, resp.)
 is the double covering of
 $T-\mathfrak{S}_T$ ($Q-\mathfrak{S}_Q$, resp.)
 branched along $\mathfrak{H}_T-\mathfrak{S}_T$ ($\mathfrak{H}_Q-\mathfrak{S}_Q$, resp.)
 and
 the divisor $\mathfrak{H}_T$ is identified with $\mathfrak{H}_Q$ under the period mapping $\Phi$ of (\ref{PhiIso}),
 we can obtain the lift $\Phi_{Q_1}: T_1 \simeq Q_1-\mathfrak{S}_{Q_1}$ of $\Phi|_{T-\mathfrak{S}_T}$.
 Next, let us consider the universal covering $T_\mathcal{D}$ of $T_1$.
 Recalling that the action of $\Gamma'$ on $\mathcal{D}-\mathfrak{S}_\mathcal{D}$ is free,
 we can see that 
 $\Phi_{Q_1}$ is lifted to $\Phi_\mathcal{D}: T_\mathcal{D} \simeq \mathcal{D}-\mathfrak{S}_\mathcal{D}.$
 So, we have the following diagram.

\begin{align}\label{diagram}
\begin{CD}
T_\mathcal{D} @> \Phi_\mathcal{D} >>\mathcal{D}-\mathfrak{S}_\mathcal{D}\\
@V  \text{free} Vp_\mathcal{D}V @VV  \text{free} V\\
T_1 @> \Phi_{Q_1} >> Q_1 -\mathfrak{S}_{Q_1}\\
@V  \mathbb{Z}/2\mathbb{Z} Vp_1V @VV  \mathbb{Z}/2\mathbb{Z}V\\
T-\mathfrak{S}_T @>\Phi| _{T-\mathfrak{S}_T}>  > Q-\mathfrak{S}_Q  \\
@V  \text{inclusion}  VV @VV  \text{inclusion}V\\
T @>\Phi > > Q  
 \end{CD}
\end{align}

Let us consider
the pull-back 
$T_\mathcal{D}^* \rightarrow T_\mathcal{D}$
of the principal $\mathbb{C}^*$-bundle
$T^*  \rightarrow T$ 
by the composition of the mapping $T_\mathcal{D} \rightarrow T_1 \rightarrow T-\mathfrak{S}_T \hookrightarrow T.$
The  isomorphism $\Phi_\mathcal{D}: T_\mathcal{D} \simeq \mathcal{D}-\mathfrak{S}_\mathcal{D}$ 
 can be lifted to 
 \begin{align}\label{liftedPer}
 \Phi_\mathcal{D^*}: T_{\mathcal{D}^*} \simeq \mathcal{D}^*-\mathfrak{S}_{\mathcal{D}^*},
 \end{align}
 where $\mathfrak{S}_{\mathcal{D}^*}$ is the preimage of $\mathfrak{S}_\mathcal{D}$ under the canonical projection.

 By the above construction of the lifts of $\Phi$,
 we can see that
 the mapping $\Phi_{Q_1}$ ($\Phi_\mathcal{D}$, $\Phi_{\mathcal{D}^*}$, resp.) is equivalent
 under the action of
 $\Gamma/\Gamma' \simeq \mathbb{Z}/2\mathbb{Z}$
 ($\Gamma$, $\mathbb{C}^* \times \Gamma$, resp.).
 Especially, 
 letting $\gamma \in O^+(A)$,
 due to Lemma \ref{LemGammaProj},
 we note that one of $\gamma$ or $-\gamma$ is an element of $\Gamma$
  and the other is not so.

 The above structure of the branches defines orbifolds.
 Such orbifolds  will be useful in Section 5.

 \section{Sections of line bundles and polynomial ring in parameters}

The natural projection
 $$\varpi: \mathcal{D}^* \rightarrow \mathcal{D}$$ 
  gives a principal $\mathbb{C}^*$-bundle.
By the period mapping $\Phi$ of (\ref{PhiIso}),
$Q=\mathcal{D}/\Gamma$ is isomorphic to the Zariski open set $T$ of the weighted projective space $\hat{T}=\mathbb{P}(4,6,10,12,18)$.
Especially,
$Q$ gives a smooth variety.
Since $\varpi$ is equivalent under the action of $\Gamma$,
we have a principal $\mathbb{C}^*$-bundle
$$
\overline{\varpi}: \mathcal{D}^*/ \Gamma \rightarrow Q.
$$

In this paper,
let $\mathcal{O}_Q (1)$ be the line bundle associated with $\overline{\varpi}$.
By the definition of associated bundles,
a section  of $\mathcal{O}_Q (1)$ 
corresponds to
a holomorphic function $\eta \mapsto s(\eta)$  ($\eta\in \mathcal{D}^*/\Gamma$) satisfying
\begin{align}\label{associatedbundle}
 s (\lambda \eta) = \lambda^{-1} s (\eta) \quad \quad (\lambda \in \mathbb{C}^*).
\end{align}
For $k\in \mathbb{Z}$,
the line bundle $\mathcal{O}_Q (1)^{\otimes k}$ is denoted by $\mathcal{O}_Q (k).$

\begin{prop}\label{PropDiagram}
There is a commutative diagram
\begin{align*}
\begin{CD}
t=(t_4,t_6,t_{10},t_{12},t_{18}) @> \text{period mapping} >> \eta \\
@V  \mathbb{C}^*\text{-action} V V @VV  \mathbb{C}^*\text{-action} V\\
\lambda^{-1}\cdot t=(\lambda^{-4}t_4, \lambda^{-6}t_6, \lambda^{-10}t_{10}, \lambda^{-12}t_{12}, \lambda^{-18}t_{18})  @> \text{period mapping} >> \lambda \eta \\
 \end{CD}
\end{align*}
Especially,
the correspondence $\mathcal{D}^*/\Gamma \ni \eta \mapsto t_j(\eta)\in\mathbb{C}$, 
via the inverse  of the period mapping $\Phi$,
defines a global section  of  the line bundle  $\mathcal{O}_Q (k)$.
\end{prop}

\begin{proof}
In Corollary \ref{CorPer}, 
the  period mapping $\Phi$ of (\ref{PerPhi})
is given by integrals of the holomorphic $2$-form $\omega_{[t]}$ of $S([t])$ ($[t] \in T$).
By virtue of the property (\ref{omegalambda}) of $\omega_{[t]}$,
we  obtain the above commutative diagram.
Together with Theorem \ref{ThmDolg} and Theorem \ref{ThmMPAP},
we can consider the inverse
 correspondence $\mathcal{D}^*/\Gamma \ni \eta \mapsto t_j(\eta)\in\mathbb{C}$
 of the period mapping.
 Considering the diagram,
 $\eta \mapsto t_k (\eta)$ defines
 a global section  of  the line bundle  $\mathcal{O}_Q (k)$
 in the sense of (\ref{associatedbundle}).
\end{proof}

Moreover, we have the following theorem.

\begin{thm}\label{Thmid}
The total coordinate ring on the smooth variety $Q$  is isomorphic to the polynomial ring of
$t_4,t_6,t_{10},t_{12}$ and $t_{18}$:
$$
\displaystyle \bigoplus_{\mathcal{L}\in {\rm Pic}(Q)} H^0 (Q, \mathcal{L}) \simeq \mathbb{C}[t_4,t_6,t_{10},t_{12},t_{18}].
$$
Here, $t_k$ gives a global section of the line bundle $\mathcal{O}_Q(k)$.
\end{thm}

\begin{proof}
Since $T$ of (\ref{TDef}) is an analytic subset of $\hat{T}=\mathbb{P}(4,6,10,12,18)$ of codimension at least  $2$,
by Hartogs's phenomenon,
we have an isomorphism
\begin{align}\label{PicT}
\iota^*_T: {\rm Pic}(\hat{T}) \simeq {\rm Pic} (T)
\end{align}
from the inclusion 
$\iota_T: T\hookrightarrow \hat{T}$.
By the isomorphism $\Phi: T \rightarrow Q$ of (\ref{PhiIso}),
we obtain
\begin{align}\label{PicPhi}
\Phi^* : {\rm Pic}(Q) \simeq {\rm Pic} (T).
\end{align}
Since $\hat{T}$ is a weighted projective space,
${\rm Pic}(\hat{T}) \simeq \mathbb{Z}$ holds.
Moreover, it holds
\begin{align}\label{totalringT}
\bigoplus_{\mathcal{L} \in {\rm Pic}(\hat{T})} H^0 (\hat{T},\mathcal{O}_{\hat{T}} (\mathcal{L}))
\simeq
\bigoplus_{k\in \mathbb{Z}} \mathbb{C}^{(k)} [t],
\end{align}
where
$\mathbb{C}^{(k)} [t] = \mathbb{C}^{(k)} [t_4,t_6,t_{10},t_{12},t_{18}]$
is the vector space of polynomials of weight $k$.
Due to (\ref{PicT}),
${\rm Pic} (Q)$ is generated by $\mathcal{O}_Q (1)$.
Together with  (\ref{PicPhi}) and (\ref{totalringT}),
we have
\begin{align}
\bigoplus_{k\in \mathbb{Z}}\mathbb{C}^{(k)}[t]
 \simeq  \bigoplus_{\mathcal{L}\in {\rm Pic} (T)}H^0 (T,\mathcal{O}_T (\mathcal{L}))
\simeq \bigoplus_{k\in \mathbb{Z}} H^0 (Q,\mathcal{O}_Q (k)).
\end{align}
Especially, $\mathbb{C}^{(k)}[t]$ gives the space
$H^0 (Q,\mathcal{O}_Q(k))$ 
of sections
of the line bundle $\mathcal{O}_Q (k)$
via the period mapping. 
\end{proof}

\begin{rem}
We have the Satake-Baily-Borel compactification
$\hat{Q}$ of $Q=\mathcal{D}/\Gamma$.
We note that 
$\hat{Q}-Q$ is an analytic subset of $\hat{Q}$
of codimension at least $2$ (see \cite{BB} Proposition 3.15).
So, by Hartogs's phenomenon again, 
the inclusion 
$\iota_Q : Q \hookrightarrow \hat{Q}$
 induces an isomorphism
$
\iota^*_Q: {\rm Pic}(\hat{Q}) \simeq {\rm Pic} (Q)
$
of Picard groups
and
the total coordinate ring on $\hat{Q}$ is given by $\mathbb{C}[t_4,t_6,t_{10},t_{12},t_{18}]$ also. 
\end{rem}

From the definition (\ref{associatedbundle}) of the associated line bundles,
a global section $s$ of $\mathcal{O}_Q (k)$ induces a holomorphic mapping $Z\mapsto \tilde{s} (Z)$  $(Z\in \mathcal{D}^*)$ satisfying the following functional equations
\begin{align*}
\begin{cases}
& \tilde{s} (\lambda Z) = \lambda^{-k} \tilde{s} (Z) \quad (\lambda \in \mathbb{C}^*), \\
&  \tilde{s} (\gamma Z) = \tilde{s} (Z) \quad (\gamma\in \Gamma) .
\end{cases}
\end{align*}
Therefore,
by recalling Definition \ref{DfModularForm},
a global section $s$ of $\mathcal{O}_Q (k)$ defines a 
modular form of weight $k$ and character ${\rm id}$ 
for the group $\Gamma=\tilde{O}^+(A)$.
Hence, we have the following corollary.

\begin{cor}\label{Corid}
Let $\Gamma=\tilde{O}^+(A).$
The inverse   of the period mapping $\Phi$ of (\ref{PerPhi})
induces an isomorphism between $\mathbb{C}[t_4,t_6,t_{10},t_{12},t_{18}]$ and the ring $\mathcal{A}(\Gamma,{\rm id}) $ of modular forms of character ${\rm id}$. 
\end{cor}

Since the argument in this section is based on the structure of the line bundles over the smooth variety $Q=\mathcal{D}/\Gamma$,
we does not consider the branch loci coming from the action of the group $\Gamma$ in this section.
However, as we considered in Section 3,
the action of $\Gamma$ on $\mathcal{D}$ is not free.
In fact,
the branch loci coming from this action
affect the structure of the ring of modular forms.
In the next section, 
we will consider the orbifolds coming from the action of $\Gamma$. 
By considering  orbibundles over the orbifolds,
we will determine the precise structure of the ring of  modular forms for $\Gamma$ with non-trivial characters.

\section{Structure of the ring of modular forms for $\Gamma$ }

In this section,
we will consider the structure of the canonical orbibundles on the orbifolds coming from the results in Section 3.3.
For the definitions and properties of orbibundles, for example, see \cite{ALR}. 
Also,  we will use the techniques of  orbifolds referring to the work  \cite{HashimotoUeda}.

We can see that the total space $\mathcal{D}^*$ of the  $\mathbb{C}^*$-bundle on $\mathcal{D}$ is a Stein manifold,
because $\mathcal{D}$ is a Stein manifold.
Also, we can see that $H^2 (\mathcal{D}^*,\mathbb{Z})=0.$
Hence, 
 we have 
$$
{\rm Pic} (\mathcal{D}^*) = 0.
$$
So, for the orbifold $\mathbb{O}=[\mathcal{D}^*/(\mathbb{C}^* \times \Gamma)]$, we have
\begin{align}\label{PicChar}
{\rm Pic} (\mathbb{O}) = {\rm Char} (\mathbb{C}^* \times \Gamma).
\end{align}
According to Theorem \ref{ThmKneser} (2) and (\ref{PicChar}),
there exists a line orbibundle corresponding to  a character of $ \mathbb{C}^* \times \Gamma$ given by 
$(\lambda,\gamma)\mapsto \lambda^{-k} {\rm det}(\gamma)^l$  $(l\in \mathbb{Z}/2\mathbb{Z})$.
In the following argument,
this bundle is denoted by
$\mathcal{O}_{\mathbb{O}} (k) \otimes {\rm det}^l$.

\begin{lem}\label{LemCano}
Let $\Omega_\mathbb{O} $ be the canonical orbibundle over the orbifold $\mathbb{O}$. 
Then,
$\Omega_\mathbb{O}\simeq \mathcal{O}_\mathbb{O}(4) \otimes {\rm det}.$
 \end{lem}

\begin{proof}
Note that
$\mathcal{D}$ is a open set of a quadratic hypersurface in $\mathbb{P}^5 (\mathbb{C})$.
Recall  that
$$\Omega_{\mathbb{P}^5(\mathbb{C})} \simeq \mathcal{O}_{\mathbb{P}^5(\mathbb{C})} (6).$$
Hence, by the adjunction formula,
we have 
\begin{align}\label{OmegaD}
\Omega_\mathcal{D} \simeq \mathcal{O}_\mathcal{D} (6-2) = \mathcal{O}_\mathcal{D} (4).
\end{align}

Next, let us consider a section of the canonical orbibundle $\Omega_\mathbb{O}$
and the action of $\Gamma$ on a holomorphic $4$-form on $\mathcal{D}$.
Let $\sigma_\delta $ 
 be a reflection for $\delta\in \Delta(A)$.
  By taking a local system of coordinates $(w_1,w_2,w_3,w_4)$ of a small open set in $\mathcal{D}$
 such that the reflection hyperplane of $\sigma_\delta$ is $\{w_1=0\}$,
 the $4$-form $dw_1\wedge dw_2\wedge dw_3\wedge dw_4$
 is changed to $-dw_1\wedge dw_2\wedge dw_3\wedge dw_4$ 
 under the action of $\sigma_\delta$.
 Together with (\ref{OmegaD}),
 it is implied that the canonical orbibundle $\Omega_\mathbb{O}$ is given by $\mathcal{O}_\mathbb{O} (4) \otimes {\rm det}.$
\end{proof}

\begin{lem}\label{LemEquiv}
The lifted period mapping $\Phi_{\mathcal{D}^*}$ in (\ref{liftedPer}) gives an isomorphism of orbifolds between $[T_{D^*}/(\mathbb{C}^* \times \Gamma)]$
and $[(\mathcal{D}^*- \mathfrak{S}_{\mathcal{D}^*})/(\mathbb{C}^* \times \Gamma) ]$
such that
$[\Phi_{\mathcal{D}^*}]^* \mathcal{O}_{[(\mathcal{D}^*- \mathfrak{S}_{\mathcal{D}^*})/(\mathbb{C}^* \times \Gamma) ]} (1) \simeq \mathcal{O}_{[T_{D^*}/(\mathbb{C}^* \times \Gamma)]} (1).$
\end{lem}

\begin{proof}
From Theorem \ref{ThmKneser} (1) and the argument of Section 3.3,
the lifted period mapping $\Phi_{\mathcal{D}}: T_\mathcal{D} \rightarrow \mathcal{D} - \mathfrak{S}_\mathcal{D}$ 
gives the monodromy covering of
$\Phi | _{T -\mathfrak{S}_T} : T-\mathfrak{S}_T \rightarrow Q-\mathfrak{S}_Q$.
Therefore,
together with the proof of Proposition \ref{PropDiagram},
we have the assertion for the mapping $\Phi_{\mathcal{D}^*}$.
\end{proof}

As far as we consider holomorphic sections of line bundles,
by virtue of  Hartogs's phenomenon,
we can ignore analytic subsets of codimension at least $2$,
such as $\mathfrak{S}_Q$ or $\mathfrak{S}_T$  in (\ref{diagram}).
In this section,
we shall often omit such subsets without notice.

\begin{prop}\label{Propd}
The degree of the polynomial $\Delta_T (t)$ in (\ref{DeltaT}) is $108$.
\end{prop}

\begin{proof}
Let us consider the canonical bundle $\Omega_{T_1}$ on $T_1$.
Since $T$ is an open subset of the weighted projective space $\hat{T}=\mathbb{P}(4,6,10,12,18)$,
we have
\begin{align}\label{OmegaT}
\Omega_T = \mathcal{O}_T (-4-6-10-12-18) =\mathcal{O}_T(-50).
\end{align}
Since the double covering $p_1:T_1\rightarrow T$ is branched along $\mathfrak{H}_{T_1}$,
by considering the holomorphic differential forms,
we have
\begin{align}\label{OmegaT1}
\Omega_{T_1} \simeq p_1^* \Omega_T \otimes \mathcal{O}_{T_1} (\mathfrak{H}_{T_1}).
\end{align}
Here, the double covering induces a morphism
$$[p_1]:[T_1/(\mathbb{Z}/2\mathbb{Z})] \rightarrow T$$
of orbifolds.
We set
$$\mathbb{T}_1 = [T_1/(\mathbb{Z}/2\mathbb{Z})] .$$
From (\ref{OmegaT}) and (\ref{OmegaT1}), we obtain
\begin{align}\label{OmegaOrb}
\Omega_{\mathbb{T}_1} \simeq [p_1]^* \Omega_T \otimes \mathcal{O}_{\mathbb{T}_1} (\mathfrak{H}_{\mathbb{T}_1})
\simeq [p_1]^* \mathcal{O}_T(-50) \otimes  \mathcal{O}_{\mathbb{T}_1} (\mathfrak{H}_{\mathbb{T}_1}).
\end{align}
We remark that the Picard group ${\rm Pic}(\mathbb{T}_1)$ 
is generated by
$[p_1]^* \mathcal{O}_T (1)$ and $\mathcal{O}_{\mathbb{T}_1} (\mathfrak{H}_{\mathbb{T}_1}).$
Letting $d$ be the degree of $\Delta_T(t)$,
due to $(\ref{T1})$,
we obtain
\begin{align}\label{DoubleRelation}
[p_1]^* \mathcal{O}_T (d) \simeq \mathcal{O}_{\mathbb{T}_1} (2\mathfrak{H}_{\mathbb{T}_1}).
\end{align}
Here, (\ref{DoubleRelation}) gives the relation which determines the structure of ${\rm Pic}(\mathbb{T}_1)$.
Hence, we have
\begin{align}\label{PicT1}
{\rm Pic}(\mathbb{T}_1) \simeq \mathbb{Z} \oplus (\mathbb{Z}/2\mathbb{Z}),
\end{align}
where
$\mathbb{Z}$ is generated by $[p_1]^* \mathcal{O}_T (1)$
and
$\mathbb{Z}/2\mathbb{Z}$ is generated by $[p_1]^* \mathcal{O}_T (-\frac{d}{2}) \otimes \mathcal{O}_{\mathbb{T}_1} (\mathfrak{H}_{\mathbb{T}_1})$.
Here, we note that $d\in 2\mathbb{Z}$
because all weights for $T$ are even.
Also,
according to the construction of (\ref{diagram}),
the torsion part $(\mathbb{Z}/2\mathbb{Z})$ in (\ref{PicT1}) should be generated by ${\rm det} \in {\rm Char}(\Gamma)$.

Since $p_\mathcal{D}:T_\mathcal{D} \rightarrow T_1$ in (\ref{diagram}) gives the universal covering,
the orbifold $[T_\mathcal{D}/\Gamma]$ is equivalent to $[T_1/(\mathbb{Z}/2\mathbb{Z})] =\mathbb{T}_1$.
Also, recall that the orbifold $[T_\mathcal{D}/\Gamma]$ is equivalent to $\mathbb{O}.$
So, the orbifold $\mathbb{T}_1$ is equivalent to $\mathbb{O}.$

Therefore, by  Lemma \ref{LemCano}, Lemma \ref{LemEquiv}  and (\ref{OmegaOrb}),
we have
\begin{align*}
\mathcal{O}_\mathbb{O} (4) \otimes {\rm det} &\simeq \Omega_\mathbb{O}
\simeq \Omega_{\mathbb{T}_1}\\
&\simeq [p_1]^* \mathcal{O}_T (-50) \otimes \mathcal{O}_{\mathbb{T}_1} (\mathfrak{H}_{\mathbb{T}_1})\\
& \simeq [p_1]^* \mathcal{O}_T \Big(-50+\frac{d}{2}\Big) \otimes \Big([p_1]^*\mathcal{O}_T \Big(-\frac{d}{2}\Big) \otimes \mathcal{O}_{\mathbb{T}_1} (\mathfrak{H}_{\mathbb{T}_1})\Big).
\end{align*}
This implies that $4=-50 +\frac{d}{2}$.
So, we obtain  $d=108$.
\end{proof}

\begin{prop}\label{PropBranch}
The branch loci of the double covering $p_1:T_1\rightarrow T$ is given by the union $\{[t]\in T \hspace{0.5mm}|\hspace{0.5mm} t_{18}=0 \} \cup \{[t]\in T \hspace{0.5mm}|\hspace{0.5mm} d_{90} (t)=0\}$  of  divisors.
There is a holomorphic function $s_9$ ($s_{45}$, resp.) on $\mathcal{D}^*$ of weight $9$ ($45$, resp.) 
satisfying
\begin{align*}
\begin{cases}
&s_9^2 =t_{18},  \\
&s_{45}^2 = d_{90} (t)  \text{ in }  (\ref{d_{90} (t)}).
\end{cases}
\end{align*}
\end{prop}

\begin{proof}
Recall that two fibres of Kodaira type $I_1$ of (\ref {Kodaira}) of elliptic surface $S([t])$ 
collapse into a  singular fibre of type $I_2$
on the divisor $\{[t]\in T \hspace{0.5mm} | \hspace{0.5mm} d_{90} (t)=0\}$ of (\ref{d_{90} (t)}).
Here, a singular fibre of type $I_2$ gives an $A_1$-singularity. 
An $A_1$-singularity acquires a  $(-2)$-curve.
(recall the argument in Section 1.2).
So,
according to the constructions of reflection hyperplanes $\mathfrak{H}_{\mathcal{D}}$ in Section 3.3 and the diagram (\ref{diagram}),
it follows that
\begin{align}\label{d_{90} (t)Delta}
\{[t]\in T \hspace{0.5mm} | \hspace{0.5mm} d_{90} (t) =0\} \subset \{[t]\in T \hspace{0.5mm} | \hspace{0.5mm}\Delta_T (t) =0\}.
\end{align}

Moreover,
according to Proposition \ref{PropCD} and Proposition \ref{PropS(t)CD},
the $K3$ surfaces $S([t])$ with the generic N\'eron-Severi lattice $U\oplus E_8(-1) \oplus E_6(-1)$  
are generically degenerated to the $K3$ surfaces with the N\'eron-Severi lattice
$U\oplus E_8(-1) \oplus E_7(-1)$
on the divisor $\{[t]\in T \hspace{0.5mm} | \hspace{0.5mm} t_{18}=0\}$.
This implies that
the divisor $\{[t]\in T \hspace{0.5mm} | \hspace{0.5mm} t_{18}=0\}$
should be contained in the image of  $\mathfrak{H}_{\mathcal{D}}$.
Namely,
\begin{align}\label{t18Delta}
\{[t]\in T \hspace{0.5mm} \hspace{0.5mm} | \hspace{0.5mm} \hspace{0.5mm} t_{18} =0\} \subset \{[t]\in T \hspace{0.5mm} | \hspace{0.5mm}\Delta_T (t) =0\}.
\end{align}

From (\ref{d_{90} (t)Delta}) and (\ref{t18Delta}),
the polynomial $t_{18} d_{90} (t)$ should divide the polynomial 
$\Delta_T (t)$. 
By virtue of Proposition \ref{Propd},
we conclude that
\begin{align}\label{t18d_{90} (t)}
\Delta _T(t ) = {\rm const. } \cdot t_{18} d_{90} (t),
\end{align}
because the polynomial $t_{18} d_{90} (t)$ is of weight $108$.

Recall the total coordinate ring on $Q$ is $\mathbb{C}[t_4,t_6,t_{10},t_{12},t_{18}]$ (Theorem \ref{Thmid}).
We had the relationship between the character ${\rm det}$ for $\Gamma$
and the double covering $p_1 :T_1\rightarrow T$ branched along the divisor $\{[t]\in T \hspace{0.5mm} | \hspace{0.5mm} \Delta_T (t)=0\}$   (recall  Section 3.3).
Moreover,
since  there is the irreducible decomposition  (\ref{t18d_{90} (t)}) of the polynomial $\Delta_T(t)$,
we have a holomorphic function $s_9$ ($s_{45}$, resp.) of weight $9$ ($45$, resp.) 
 on $\mathcal{D}^*$
satisfying the relation
$
s_9^2=t_{18}$
($
  s_{45}^2 =d_{90} (t),
$
resp.).
\end{proof}

We note that $s_9$ ($s_{45}$, resp.) does not vanish on the  divisor corresponding to the divisor $\{d_{90}=0 \}$ ($\{t_{18}=0\}$, resp.).
So, $s_9$ and $s_{45} $ are not modular forms of character {\rm det} in the sense of Definition  \ref{DfModularForm}.
However, we have the following theorem.

\begin{thm}\label{ThmTotal}
The holomorphic function $s_{54} =s_9 s_{45}$ on $\mathcal{D}^*$ gives a modular form of weight $54$ and character ${\rm det}$. 
 The relation $s_{54}^2= t_{18} d_{90} (t)$ determines the ring $\mathcal{A}(\Gamma)$. 
Here, $\mathcal{A}(\Gamma) = \mathcal{A}(\Gamma,{\rm id}) \oplus \mathcal{A}(\Gamma,{\rm det})$ and
\begin{align*}
\begin{cases}
& \mathcal{A}(\Gamma,{\rm id}) = \mathbb{C}[t_4,t_6,t_{10},t_{12},t_{18}],  \\
& \mathcal{A} (\Gamma,{\rm det}) =s_{54} \mathbb{C}[t_4,t_6,t_{10},t_{12},t_{18}] .
\end{cases}
\end{align*}
\end{thm}

\begin{proof}
The argument in Section 3.3 implies that 
  modular forms for $\Gamma$ of character ${\rm det}$ 
are corresponding to the double covering $p_1$.
Also,
such modular forms must vanish on  
the preimage of $\mathfrak{H}_\mathcal{D}$ under the canonical projection $\mathcal{D}^* \rightarrow \mathcal{D}$.
From Proposition \ref{PropBranch}, 
the branch loci is given by the union of the two divisors  $\{t_{18}=0\}$ and $\{d_{90}(t)=0\}$.
Therefore, every modular form of character ${\rm det}$ is given by a product of $s_{54}=s_9 s_{45}$ and $F(t) \in \mathbb{C}[t_4,t_6,t_{10},t_{12},t_{18}]$. 
\end{proof}

\begin{rem}\label{RemDiscriminant}
The classical  elliptic discriminant form, which is of weight $12$ and generates the vector spaces of cusp forms for $SL(2,\mathbb{Z})$,
is originally given 
by the discriminant of the right hand side of the Weierstrass equation of an elliptic curve.
Also, as in Section 2.3,
from the discriminant of the right hand side of the defining equation (\ref{SCD}) of the $K3$ surface $S_{CD}$,
we can compute  Siegel modular forms of non-trivial character,
which determine the structure of the ring of  Siegel modular forms.
Our constructions of $s_9$ and $s_{45}$ are counterparts of such  constructions from discriminants of Weierstrass forms.
\end{rem}

Due to Proposition \ref{PropSiegelModular}, Proposition \ref{PropS(t)CD} and Remark \ref{RemDiscriminant}, 
our family of $K3$ surfaces can be regarded as a natural extension of the family of (\ref{SCD})
and
our modular forms for the lattice $A$ give natural counterparts of the classical Siegel modular forms 
from based on inverse period mappings of $K3$ surfaces.
The author expects that 
we can study our modular forms more thoroughly as follows.

\begin{itemize}

\item
There are several  results  unifying periods of $K3$ surfaces and Shepherd-Todd groups.
For example, in the previous paper \cite{NCF},
we obtained a unification of arithmetic properties of Hilbert modular forms,
the invariants of the Shepherd-Todd group of No. 23 
and the geometry of the $K3$ surfaces of \cite{NTheta}.
By the way,  \cite{DK} or \cite{FS} studied
modular forms related to the Shepherd-Todd group of No.33.
However, to the best of the author's knowledge,
there has not been an explicit expression of their modular forms 
from the moduli of algebraic varieties.
The author expects that
our results will provide a natural moduli interpretation of their modular forms 
and enable us to connect arithmetic properties  of modular forms and our $K3$ surfaces.
It will be studied in the forthcoming research paper with H. Shiga.

\item 
Recently, among researchers studying the Langlands program,
it is more and more important to study 
modular forms on the real $3$-dimensional hyperbolic space $\mathbb{H}_{\mathbb{R}}^3$.
Since they are real analytic functions,
it is highly non-trivial to construct and study them explicitly.
Matsumoto, Nishi and Yoshida \cite{MNY}
succeeded to obtain
 modular forms on $\mathbb{H}_{\mathbb{R}}^3$
via the period mapping for $K3$ surfaces of \cite{MSY}.
However, 
their results are much complicated.
Here, 
our modular forms are also defined 
on the $4$-dimensional  symmetric domain of type $IV$.
Recall that we have a merit that
 our modular group is simpler than that of \cite{MSY}.
 So, the author expects that 
 we can apply our results to construct modular forms on $\mathbb{H}_{\mathbb{R}}^3$ effectively.
\end{itemize}

\section*{Acknowledgment}
The author would like to thank Professor Hironori Shiga for helpful advises and  valuable suggestions.
He  thanks Professor Manabu Oura for suggesting me the results of \cite{ST}.
Moreover, 
he would like to thank Professor Mehmet Haluk Seng\"{u}n
for letting me know the paper \cite{MNY}.
This work is supported by JSPS Grant-in-Aid for Scientific Research (18K13383)
and 
MEXT Leading Initiative for Excellent Young Researchers (2018-18).

\begin{center}
\hspace{7.7cm}\textit{Atsuhira  Nagano}\\
\hspace{7.7cm}\textit{Faculty of Mathematics and Physics}\\
 \hspace{7.7cm} \textit{Institute of Science and Engineering}\\
\hspace{7.7cm}\textit{Kanazawa University}\\
\hspace{7.7cm}\textit{Kakuma, Kanazawa, Ishikawa}\\
\hspace{7.7cm}\textit{920-1192, Japan}\\
 \hspace{7.7cm}\textit{(E-mail: atsuhira.nagano@gmail.com)}
  \end{center}

\end{document}